\documentclass[a4paper, 11pt,  dvipsnames]{amsart}
\usepackage[utf8]{inputenc}
\usepackage[T1]{fontenc}
\usepackage{AlegreyaSans}
\usepackage{Alegreya}
\usepackage[english]{babel}
\usepackage{amsmath, amssymb, amsthm, amsfonts, mathtools, nicefrac, mathabx}
\usepackage{dsfont} 
\usepackage{amsbsy} 
\usepackage{floatflt} 
\usepackage{etoolbox}
\usepackage{pdfpages}
\usepackage{amsrefs}
\usepackage{graphicx} 
\usepackage{xcolor}
\usepackage{extarrows}
\usepackage[all]{xy}
\usepackage{tikz}
\usetikzlibrary{arrows}
\usepackage{tikz-cd}
\usepackage{mathtools}
\usepackage[format=plain,margin=30pt,labelfont=bf, labelsep=space]{caption} 
\usepackage[left=2.9cm, right=2.9cm, bottom=3cm]{geometry}
\usepackage{hyperref} 

\hypersetup{
	colorlinks=true,
	linktocpage=true,
	linkcolor=[RGB]{200,0,100},
	citecolor=[RGB]{100,0,200},      
	urlcolor=cyan,
}

\usepackage[noabbrev,capitalize]{cleveref} 

\theoremstyle{plain}
\newtheorem{thm}{Theorem}

\newtheorem*{thm*}{Theorem}
\newtheorem{lem}[thm]{Lemma}
\newtheorem{prop}[thm]{Proposition}
\newtheorem{cor}[thm]{Corollary}
\newtheorem*{cor*}{Corollary}
\theoremstyle{definition}
\newtheorem{defn}[thm]{Definition}
\newtheorem{exmp}[thm]{Example}
\newtheorem*{exmp*}{Example}

\theoremstyle{remark}
\newtheorem*{rem*}{Remark}
\newtheorem{rem}[thm] {Remark}

\AfterEndEnvironment{defn}{\noindent\ignorespaces}
\AfterEndEnvironment{lem}{\noindent\ignorespaces}
\AfterEndEnvironment{thm}{\noindent\ignorespaces}
\AfterEndEnvironment{thm*}{\noindent\ignorespaces}
\AfterEndEnvironment{cor}{\noindent\ignorespaces}
\AfterEndEnvironment{cor*}{\noindent\ignorespaces}
\AfterEndEnvironment{exmp}{\noindent\ignorespaces}
\AfterEndEnvironment{exmp*}{\noindent\ignorespaces}
\AfterEndEnvironment{rem}{\noindent\ignorespaces}
\AfterEndEnvironment{rem*}{\noindent\ignorespaces}
\AfterEndEnvironment{proof}{\noindent\ignorespaces}
\AfterEndEnvironment{prop}{\noindent\ignorespaces}

\newcommand{\R}{{\mathbb R}}
\newcommand{\C}{{\mathbb C}}
\newcommand{\Z}{{\mathbb Z}}
\newcommand{\N}{{\mathbb N}}
\newcommand{\Q}{\mathbb{Q}}

\newcommand{\cA}{\mathcal{A}}

\newcommand{\cO}{\mathcal{O}}

\newcommand{\cX}{\mathcal{X}}

\newcommand{\supp}{\operatorname{supp}}

\newcommand{\sym}{\operatorname{\mathrm{Sym}}}
\newcommand{\prob}{\operatorname{Prob}}

\title{Random operators, spectral measures, and local empirical convergence in sofic groups}
\author{Miguel Donoso-Echenique}
\author{Felix Pogorzelski}
\author{Michael Schr\"odl-Baumann}

\address{Miguel~Donoso-Echenique, Fachbereich 10 Mathematik und Informatik, Universit\"at M\"unster}
 \email{mdonosoe@uni-muenster.de, migueldonosoe@gmail.com}

\address{Felix~Pogorzelski, Institut f\"ur Mathematik, Universit\"at Leipzig
\\ and  
Israel Institute of Advanced Studies (IIAS), Jerusalem}
\email{felix.pogorzelski@math.uni-leipzig.de}

\address{Michael Schr\"odl-Baumann}
\email{mschroedlbaumann@gmail.com}

\begin{document}

\begin{abstract}
    In this paper, we consider the problem of approximating the spectral distribution for a class of random operators over sofic groups. For this purpose, we make use of  the concept of locally and empirically converging measures defined by Austin.  
    We establish weak convergence of the density of states measures along random finite-volume analogs. For operators taking finitely many rational values, we prove a L\"uck type approximation theorem yielding pointwise convergence of the spectral measures. In the wider context of arbitrary complex coefficients, we show pointwise convergence of the spectral distribution functions along adapted approximants with varying rational coefficients. Our results apply  to the class of periodically approximable groups as defined by Bowen. More generally, we show that every invariant probability measure on a finite-state configuration space that arises as a weak-$\ast$ limit of periodic measures admits an approximation in the local and empirical sense. 
\end{abstract}

\maketitle

\tableofcontents

\thispagestyle{empty}

    \section{Introduction and main results}

    \subsection{Background of results}
    Discrete random Schr\"odinger  operators are key mathematical objects to describe quantum mechanical phenomena. The past decades have witnessed far reaching developments in various directions. We refer the reader to the monographs and surveys \cites{CL90,KM07,Kir08,DF1,DF2} and references therein for in-depth information. 

    For operators in geometries with enough symmetry, random Schr\"odinger operators belong to a class of equivariant families of measurable operators that are elements of a von Neumann algebra. As such, their spectral distribution, measuring the portion of spectral mass below a prescribed energy level, can
    be defined by a trace. This leads to the notion of the integrated density of states (IDS), defined in analogy to the  
 Pastur-Shubin trace formula referring to  work of Pastur \cites{Pas71,Pas72,Pas73} and Shubin \cite{Shu79} on the spectral distribution of random operators in both continuous and discrete settings. More investigations in that direction followed soon after, see e.g.\@ \cites{Fuk74,FNN75,FN76,Nak77,Kot77,KM82c,AS83,CS83,BLT85,Bel86}. 
 
 The IDS and the underlying density of states measure are central quantities in  operator theory, as they may contain spectral information for instance on potential gap labels, spectral type, Lifshitz tails and localization phenomena, see the  monographs and papers cited above.

For discrete operators, the underlying Hilbert space is typically modeled over a graph with its vertices as potential positions of a particle, and edges encode potential change of positions.
An important problem is the question of approximability of the IDS by normalized eigenvalue counting functions of finite dimensional analogs, backing the interpretation of the IDS measuring the normalized number of states below a fixed energy level. Here one might investigate weak convergence of spectral measures or convergence in a strong form, where the spectral distribution functions converge uniformly over all energy values. 
 
 In the setting of Cayley graphs and deterministic operators, this strong variant can be phrased as a  L\"uck type approximation problem, going back to the seminal paper of L\"uck \cite{Luc94} on approximation of $\ell^2$-Betti numbers by ordinary Betti numbers. Related approximation results for $\ell^2$-invariants can for instance be found in \cites{Fab98,DM98,Eck99}. 
 This naturally raises the issue of approximating the normalized eigenspace dimensions of a group ring element by eigenvalue counting functions associated with finite dimensional analogs over sofic approximations.  This form of the strong approximation conjecture has been verified for amenable groups by Elek \cite{Ele06}, generalizing the result from \cite{Dod+03} on torsion-free elementary amenable groups. For group ring elements with algebraic coefficients, the strong form of approximation was proven by Thom in \cite{thom},
 thus generalizing results from \cite{Dod+03} and also
 settling in the affirmative the algebraic eigenvalue conjecture stated in \cite{Dod+03}. The L\"uck approximation theorem for arbitrary coefficients was proven by Jaikin-Zapirain in \cite{Jai19}. 
 Pointwise convergence of the spectral measures, and therefore also uniform convergence of the spectral distribution functions along graph sequences
 has been verified in certain situations as well, see \cites{MY02, MSY03, Ele08,Pog14} for hyperfinite sequences  or \cite{abert-thom-virag} for operators with integer coefficients. A result on continuity of $\ell^2$-Betti numbers with respect to Benjamini-Schramm convergence of simplicial complexes has been obtained in \cite{schrodl2020}. In the recent paper \cite{LP24}, the authors study numerical aspects on convergence in residually finite groups.

 For random operators over amenable groups, the strong form of convergence along restrictions on F{\o}lner sequences has been 
verified in many situations without any assumptions on the underlying coefficients, see e.g.\@ \cites{LS05, GLV07, LMV08, LV09, LSV11, LSV12, Sch12, PSS13, PS16, SSV20}. Considerable progress towards sofic Cayley graphs was made by Schumacher and Schwarzenberger \cite{SS15}, who proved weak convergence of the IDS along sofic approximations for (potentially unbounded) random Hamiltonians with suitable stochastic independence properties. 

The goal of the present paper is to extend the dynamical scope of IDS convergence to general stationary and ergodic random processes over sofic groups. In particular, we aim at models beyond amenable structures and beyond i.i.d.\@ models, 
thus being able to cover periodic and aperiodic random models without the presence of a classical ergodic theorem.

To this end, we apply the notions of {\em locally weak-$*$-convergence ($lw^{*}$-convergence)}, as well as {\em local and empirical convergence ($le$-convergence)} of probability measures on configuration spaces over sofic approximations 
developed in the context of sofic entropy by Austin \cite{Aus16}. (In the latter paper, $le$-convergence was originally introduced under the name {\em quenched convergence}, but the terminology was changed by Austin himself a bit later, see \cite{Aus19} and the explanations on p.~2 therein.) Roughly speaking, these convergences reflect a soficity phenomenon, and  an ergodicity phenomenon for random processes, respectively. For a large class of equivariant operators specified in Section~\ref{sec:weakconvergence}, called {\em continuous operators} below, this allows one to define random finite analogs in a natural fashion (and the approximating sequence will be called an {\em induced approximation} below), and to pursue convergence in probability with respect to the $le$-convergent sequence of measures. The main results of our work are the following.
\begin{itemize}
    \item We study the class of so-called {\em  PA (periodically approximable) groups} and show that every law of an ergodic random process over a finite state space admits an $le$-convergent approximation, see Corollary~\ref{cor:PAEPA}. 
    \item For a general class of random operators over Cayley graphs of sofic groups, we  prove weak convergence of the spectral measures corresponding to  elements in an induced approximation, see Theorem~\ref{lemma-weakconvergence}.
    \item Going beyond weak convergence, we employ the techniques from \cite{abert-thom-virag} to prove a L\"uck-type convergence result for random operators taking finitely many rational coefficients, see Theorem~\ref{thm:MAIN}.
    \item Using monotone operator convergence for traces in von Neumann algebras, we prove approximation of the IDS in all energy values for operators taking finitely many complex coefficients via adapted finite-volume analogs with varying rational coefficients, see Corollary~\ref{cor:MAIN_approx}.  
\end{itemize}

\subsection{Description of results}
Taking random Schr\"odinger operators as a guiding example, we are now going to explain the aforementioned results in more detail. For pedagogical reasons, some notions will only be explained here in a heuristic manner. We refer to the Sections~\ref{sec:prelim} and Section~\ref{sec:soficconvergence} for the rigorous definition of all key players. 

\medskip
Fix a sofic group $G$ generated by a finite and symmetric set $S$. We denote the arising Cayley graph by $\mathrm{Cay}(G,S)$. The notion of {\em sofic groups}  was introduced by Gromov \cite{Gro99}, the terminology was coined by Weiss \cite{weiss}. Roughly speaking, a group is sofic if there is a sequence $(V_n,E_n)$ of finite graphs with their egdes $E_n$ labeled by the elements in $S$ such that asymptotically, the local graph structure about vertices in $V_n$ matches the local graph structure about the identity vertex $e$ in $\mathrm{Cay}(G,S)$. It is not known whether there is a non-sofic group. Fix further a compact metric space $(\mathcal{X},d)$. Then there is a natural continuous translation action of $G$ on the configuration space $\mathcal{X}^G$, endowed with the product topology. For a continuous function $F:\mathcal{X} \to \R$, a family $H^{\omega}_F:\ell^2(G) \to \ell^2(G)$ of bounded self-adjoint operators can be defined by 
\[
H_F^{\omega}u(x) = \Delta u(x) \, + \, F\big(\omega(x)\big) u (x), \quad u \in \ell^2(G), \quad x \in G, \quad \omega \in \mathcal{X}^G,
\]
where $\Delta$ denotes the combinatorial graph Laplacian. Together with a translation-invariant and  Borel probability measure $\mu$ over $\mathcal{X}^G$, we obtain a {\em $G$-process} $(G,\mathcal{X},\mu)$. We then refer to $\{H^{\omega}\}$ as a {\em random Schr\"odinger operator} over $\mathrm{Cay}(G,S)$ with potential function $F$ (and law $\mu$). In the situation where $\mu$ is ergodic, it makes sense to define its IDS via a Pastur-Shubin type formula,  
\[
N_F(\beta) = \int_{\mathcal{X}^G} \langle \delta_{e},\, 1_{]-\infty,\beta]}(H_F^{\omega})\delta_{e} \rangle\, d\mu(\omega), \quad \beta \in \R,
\]
where $1_{]-\infty,\beta]}$ denotes the spectral projection operator corresponding to the energy levels below $\beta$.  As for testing approximation by finite-volume analogs, it is natural to take a sofic approximation $(V_n,E_n)$ and to define the operators $H^{\rho}_{F,n}: \ell^2(V_n) \to \ell^2(V_n)$ acting as 
\[
H^{\rho}_{F,n}u(v) = \Delta u(x) \, + \, F\big(\rho(v)\big)u(v), \quad u \in \ell^2(V_n), \quad v \in V_n, \quad \rho \in \mathcal{X}^{V_n}, 
\]
where with a slight abuse of notation, $\Delta$ now also denotes the combinatorial Laplacian on $(V_n, E_n)$. Given $\beta \in \R$, we define 
\[
N^{\rho}_{F,n}(\beta):= \max\{i \leq |V_n|\, |\, \lambda_i \leq \beta\},
\]
where $(\lambda_i)_{i=1}^{|V_n|}$ is an increasing enumeration of all eigenvalues of $H_{F,n}^{\rho}$. Thus, $N^{\rho}_{F,n}(\beta)$ counts all eigenvalues less or equal than $\beta$, including multiplicities.  Clearly, the normalized versions of these function can only be suitable approximants for the IDS if the $\rho_n \in \mathcal{X}^{V_n}$ are carefully chosen, taking into account the law $\mu$ of the process. This will be taken care of by a sequence $(\mu_n)$ of probability measures on $\mathcal{X}^{V_n}$ locally and empirically converging to $\mu$ as defined by Austin in \cites{Aus16,Aus19}. The latter builds on locally weak-$*$ convergence ($lw^{*}$-convergence), which can be interpreted as a soficity phenomenon of the measure $\mu$. In the special situation where $\mathcal{X} = \mathcal{A}$ is a finite set (endowed with the discrete metric), $lw^{*}$-convergence is tightly related to Benjamini-Schramm convergence of graphs \cite{BS01} with their vertices being additionally decorated by finitely many labels. Accordingly, we refer to $(V_n,E_n,\mu_n)$ as a {\em sofic model} for the process $(G,\mathcal{X},\mu)$. In case $\mu$ is additionally ergodic, the convergence of $(\mu_n)$ is even in the local and empirical sense. This roughly means the validity of an ergodic type theorem for continuous functions $f:\mathcal{X}^G \to \C$, in that for large $n \in \N$,  
with high probablitiy $\mu_n$, one has
\begin{align} \label{eqn:ergodicheuristics}
\frac{1}{|V_n|} \sum_{v \in V_n} f\left( \Pi_v(\rho) \right) \, \, \approx \, \, \int_{\mathcal{X}^G} f(\omega)\, d\mu(\omega), \quad \rho \in \mathcal{X}^{V_n}, 
\end{align}
where $\Pi_v: \mathcal{X}^{V_n} \to \mathcal{X}^G$ denotes a natural lift of the configuration $\rho$ with center $v$ into a configuration in $\mathcal{X}^G$. Such sofic models exist for various classes of $G$-processes, such as Bernoulli systems \cite{Aus16}*{Lemma~5.2 and Corollary~5.8}, and (for a  slightly reformulated notion) also for certain algebraic actions or actions on profinite groups \cite{Hay19}*{Proposition~1.2}. Another source are so-called {\em PA groups}, where PA stands for {\em periodically approximable}. The latter were introduced by Bowen \cite{Bow03} as groups with the property that every law $\mu$ of a finite-state process can be weak-$*$-approximated by periodic probability measures, i.e.\@ by invariant probability measures with finite support. Kechris studied {\em property MD} in  \cite{Kec12}, a notion equivalent to PA. A strenghtening is the notion of  EPA groups for which the ergodic periodic measures are weak-$*$ dense in the space of invariant measures. To the best of our knowledge it is  an open problem whether the properties PA and EPA are equivalent.   
For more information on connections between these different notions we refer to the survey \cite{BK20}. The class of PA groups contains all finite rank free groups and all   residually finite amenable groups. We refer to Section~\ref{sec:PA} for further details. 

\medskip

Denote by $\mathcal{A}$ a finite set and define by $Q$ the largest residually finite quotient of $G$ with canonical homomorphism $\pi:G \to Q$. We write $\mathcal{P}(\mathcal{A}^G)$ for the set of all periodic probability measures on $\mathcal{A}^G$. As for the existence of sofic models, we obtain the following for PA groups.

\medskip

{\bf Proposition A.} {\em  If $Q$ has property PA, then the push forward under $\pi$ of the space of $Q$-invariant probability measures on $\mathcal{A}^Q$ is equal to the weak-$*$-closure $\overline{\mathcal{P}(\mathcal{A}^G)}$. In particular, every process $(G,\mathcal{A},\mu)$ over a PA group $G$ admits a sofic model $(V_n,E_n,\mu_n)$.}

\medskip
We will see further that for ergodic processes, the sequence $(\mu_n)$ obtained in the proposition above le-converges to $\mu$. 
The proof of the proposition are given in Proposition~\ref{prop:EPA_quotient} and Corollary~\ref{cor:PAEPA}. Since PA groups are necessarily residually finite, one can even show that the sofic approximations arise from group homomorphisms. As an example, many Baumslag-Solitar groups have residually finite quotients that have property PA, see Example~\ref{exa:BSgroups}. 

\medskip
We return to the issue of spectral approximation, making use of the relation~\eqref{eqn:ergodicheuristics}.
We write $N_F^{\prime}$ and $N^{\rho\,\prime}_{F,n}$ for the normalized
densitiy of states measures associated with $N_F$ and $N^{\rho}_{F,n}$, respectively. We prove the following result on weak convergence.

\medskip
{\bf Theorem B.} {\em Let $(V_n, E_n,\mu_n)$ be a sofic model for an ergodic $G$-process $(G,\mathcal{X},\mu)$. Then for every $\varepsilon > 0$ and every bounded continuous function $f:\R \to \R$, we have 
\[
\lim_{n \to \infty} \mu_n \big( \big\{ \rho \in \mathcal{X}^{V_n}:\, |N_{F,n}^{\rho\,\prime}(f) - N_F^{\prime}(f)| < \varepsilon \big\} \big) =1. 
\]
}

\medskip
The proof will be given in Theorem~\ref{lemma-weakconvergence}, which applies to the larger class of continuous operators and their induced approximations, see Section~\ref{sec:weakconvergence}. We briefly compare our results with the findings in \cite{SS15}. We point out that the above convergence can be realized for ergodic processes without further independence or mixing properties. For random Hamiltonians and sofic approximations randomized by an i.i.d.\@ process, Schumacher and Schwarzenberger obtained  in \cite{SS15} almost sure convergence.
In contrast to the continuous operator considered in this paper, this also includes a certain class of unbounded operators. Related results on weak convergence of the density of states measure with respect to Benjamini-Schramm convergence can be found in \cites{abert-thom-virag,AS19,BC14,BGM24,BL10,BSV17,Ele08,PogPhD}. For an anlogous result in the context of convergence of Delone dynamical systems, see \cite{BP20}. For Schr\"odinger operators with i.i.d.\@ potential on the Canopy graph, an approximation along exhaustive finite sets has been given in \cite{AW06}. {An interesting follow-up question is the possibility of convergence of spectra (as sets) in sofic or residually finite models. Under convergence of topological dynamical systems, this has been investigated in depth in \cite{BBdN18}.}

\medskip

Towards a strong variant of spectral approximation, we prove the following for potential functions with rational values. 

\medskip
{\bf Theorem C.} {\em Let $\mathcal{X}=\mathcal{A}$ be finite and let $(V_n,E_n,\mu_n)$ be a sofic model for the ergodic $G$-process $(G,\mathcal{A},\mu)$. Let $F:\mathcal{A} \to \Q$ be a map. 
Then there is a sequence $(A_n)$, $A_n \subseteq \mathcal{A}^{V_n}$ with $\lim_n \mu_n(A_n)=1$ such that for every sequence $(\rho_n)$ with $\rho_n \in A_n$, we obtain 
 \begin{align*} 
	\lim_{n \to \infty} \sup_{\beta \in \R} \left| \frac{N_{F,n}^{\rho_n}(\beta)}{|V_n|} - N_F(\beta) \right| = 0.
\end{align*}
}

\medskip

The general statement for continuous operators taking finitely many rational values can be found in Theorem~\ref{thm:MAIN} together with Corollary~\ref{cor:rationalvalues}. Their proofs build on ideas from \cites{Luc94,thom} in the very form given in \cite{abert-thom-virag}. For amenable groups and restrictions of random operators on F{\o}lner averages, uniform convergence of the IDS has been verified for arbitrary coefficients in various contexts, see e.g.\@ \cites{LS05, GLV07, KSSV26,LMV08, LV09, LSV11, LSV12, Sch12, PSS13, PS16, SSV20}. 

\medskip
The situation for arbitrary real or complex coefficients is still open and we do not know whether the techniques from \cite{Jai19} can be extended to the random case. Still, we can approximate the IDS by adapted sofic models with varying rational coefficients. The following specifies to random Schr\"o\-dinger operators over PA groups. Again, we assume that $\mathcal{X} = \mathcal{A} $ is finite. For a potential $F:\mathcal{A} \to \R$, consider a monotonically increasing sequence $(F_n)$ of maps $F_n:\mathcal{A} \to \Q$ converging pointwise to $F$ as $n \to \infty$.  

\medskip
{\bf Theorem D.} {\em Let $G$ be a PA group together with an ergodic $G$-process $(G,\mathcal{A},\mu)$. Then there is a sofic model $(V_n,E_n,\mu_n)$ and a sequence $(A_n)$ with $A_n \subseteq \mathcal{A}^{V_n}$, $\lim_n \mu_n(A_n)=1$ such that for every sequence $(\rho_n)$ with $\rho_n \in A_n$, we have 
\[
\lim_{n \to \infty} \frac{N^{\rho_n}_{F_n,n}(\beta)}{|V_n|} = N_F(\beta) \quad \mbox{ for all } \quad \beta \in \R. 
\]
}

\medskip

The assertion is given in Corollary~\ref{cor:Schrodinger} and its proof is based on Theorem~\ref{thm:OPMON} and Corollary~\ref{cor:MAIN_approx}, applying to general continuous operators. A crucial viewpoint is to consider the IDS as a trace in a von Neumann algebra as constructed in \cite{LPV07}. This allows one to exploit monotonicity of the IDS with respect to the underlying trace as outlined in \cite{BK90}. Note that in general, uniform convergence in $\beta$ cannot be expected, see Remark~\ref{rem:notuniform}.

\subsection*{Organization of the paper} The Section~\ref{sec:prelim} is devoted to preliminary information on graphs, $G$-processes, bounded self-adjoint operators and their spectral calculus. The central notions of $lw^{*}$-convergence and $le$-convergence are introduced in Section~\ref{sec:soficconvergence}
after a brief discussion of so\-fi\-ci\-ty of groups. We provide a detailed study of PA groups and quotients in Section~\ref{sec:PA}, and show that finite-state space processes admit residually finite models. The Section~\ref{sec:weakconvergence} is devoted to weak convergence of the spectral measure for induced approximations of continuous operators. The strong approximation statement for finitely many rational coefficients is proven in Section~\ref{sec:luck}. For finitely many coefficients which may be complex, we prove  in Section~\ref{sec:monotoneOP} pointwise convergence of the IDS  for finite-volume analogs with varying rational coefficients. This result makes use of a Gershgorin type criterion on positive (semi-)definiteness of bounded self-adjoint operators explained in a supplementary Section~\ref{sec:supplementary}.

\subsection*{Acknowledgements} FP and MSB exchanged first ideas leading to this paper during the SMS-DMV Workshop on Groups and Dynamics: Spectra and $L^2$-invariants (Geneva, 2018), and are grateful to the meeting's organizers, in particular to Daniel Lenz for insightful discussions.
MSB thanks the Institut für Algebra und Geometrie, KIT Karlsruhe and the Institut für Mathematik, Universität Leipzig, for support that enabled a one-week visit to FP in Leipzig in 2019, during which the idea for this paper was further developed. 
FP thanks Dominik Kwietniak for an invitation to the Jagellonian University in Krak{\'o}w and for enlightening discussions about local and empirical convergence. 
Moreover, FP is grateful to Melchior Wirth for a conversation on von Neumann algebras and in particular, for pointing out the reference \cite{BK90}. The authors are grateful to valuable comments by Tim Austin, Siegfried Beckus and Ivan Veseli{\'c}
on a first draft of the paper. The manuscript was finished during a stay of FP at the Israel Institute for Advanced Studies (IIAS) in Jerusalem, and FP thanks the institute for the excellent working conditions. FP acknowledges partial support through the  German-Israeli Foundation for Scientific Research and Development (GIF) and the Deutsche Forschungsgemeinschaft (DFG). MDE was funded by the Deutsche Forschungsgemeinschaft (DFG, German Research Foundation) under Germany's Excellence Strategy during the last stages of this work.



    \bigskip

	\section{Preliminaries}~\label{sec:prelim}

    We introduce the objects, terminology and concepts needed for this paper.
	
	\subsection{Graphs}
	 A {\em graph} is a pair $(V,E)$, where $V$ is a countable and discrete set and $E \subseteq V \times V$ is a symmetric set, i.e.\@ $(x,y) \in E$ implies $(y,x) \in E$. The elements in $V$ are called {\em vertices} and the elemtents in $E$ are called {\em edges}. In $(x,y) \in E$ for some $x,y \in V$ we say that $x$ and $y$ are {\em adjacent} or synonymously, they {\em share an edge}, and we use the notation $x \sim y$.  An {\em edge-labeling} by a set $S$ is a map $\vartheta:E \to S$ and a graph with an edge-labeling is called {\em edge-labeled}. 
	 Note that every graph without an edge labeling can be treated as a graph with a constant edge labeling map $\vartheta:E \to \{\ast\}$.
	A {\em path} in $(V,E)$ is a finite tuple $(x_n) \in V^N$ for some $N \in \N$ with the property that $x_i \sim x_{i+1}$ for all $1 \leq i \leq N-1$. We call $N$ the {\em length} of the path and say that $(x_n) \in V^N$ is a {\em path of length $N-1$}.
	We further call two vertices $y,z \in V$ {\em connected by a path} if there is an $N \in \N$ and a path $(x_n) \in V^N$ with $x_1 = y$ and $x_N = z$. 	The graph $(V,E)$ is {\em connected} whenever each pair of distinct vertices is connected by a path. For a vertex $x \in V$, we define the {\em vertex degree $\mathrm{deg}(x)$ of $x$} by the cardinality of vertices $y \in V$ such that $x$ and $y$ are adjacent. The graph $(V,E)$ is said to be {\em bounded} if $\sup_{x \in V} \mathrm{deg}(x) < \infty$. For $W \subseteq V$, the {\em induced subgraph by $W$} is the graph with vertex set $W$ and edge set $E_W = E \cap (W \times W)$. If $(V,E)$ has an edge labeling $\vartheta:E \to S$, then  $(W,E_W)$ also inherits an edge labeling $\vartheta_W := \vartheta_{|E_W}$. 
	With a slight abuse of notation, we also write $W$ for the induced graph (and not just for its vertex set). 
	For two graphs $(V_1,E_1)$ and $(V_2, E_2)$ with edge labelings $\vartheta_1$ and $\vartheta_2$ into the same set $S$, 
 we say that the graphs are {\em isomorphic} if there is a bijective map $\varphi:V_1 \to V_2$ preserving the vertex- and edge relations, as well as the edge labelings. Any such map $\varphi$ is called a {\em graph isomorphism}. 
	In this situation, we write $(V_1,E_1) \cong (V_2,E_2)$ or $(V_1, E_1) \stackrel{\varphi}{\cong} (V_2, E_2)$ in order to emphasize the reference to the graph isomorphism $\varphi$. Again, if the reference to the edge sets are clear we also use the abbreviated notations $V_1 \cong V_2$ and $V_1 \stackrel{\varphi}{\cong} V_2$. 
	The {\em path disctance} on a graph $(V,E)$ is defined as $	d_E:V \times V \to \N_0 \cup \{\infty\}$, where
	\[
d_E(x,y) = \inf\big\{ N \in \N: \mbox{ there is a path of length } N \mbox{ connecting } x \mbox{ and } y\big\},
	\]
	for  $x \neq y$ and $d_E(x,x) = 0$ for all $x \in V$.
	We use here the convention that $\inf \emptyset = \infty$. Note that for a connected graph $(V,E)$, $d_E$ defines a metric on $V$, called the {\em path metric}. Given $R > 0$ and $x \in V$, we write $B_E(x,R)$ for the (closed) ball of radius $R$ around $x$, defined as 
	\[
	B_E(x,R) := \big\{ z \in V:\, d_E(x,z) \leq R \big\}.
	\]
	Finally we say that a graph $(V,E)$ is {\em finite} if $V$ has finite cardinality.

	\subsection{Cayley graphs}
		The most important example for us arises from  finitely generated groups $G$ with (finite) symmetric generating set $S \subseteq G \setminus \{e\}$, where $e$ will always denote the identity element in $G$. Recall that $S \subseteq G$ is symmetric if $S=S^{-1}= \{s^{-1}:s \in S\}$.
  The {\em (left) Cayley graph} $\mathrm{Cay}(G,S)$ on $G$ with respect to $S$ is the graph with vertex set $V=G$ and $x,y \in G$ share an edge if and only if $y=sx$ for some $s \in S$. Note that the edge set is symmetric  due symmetry of $S$. This way we obtain a graph labeling by setting $\vartheta:E \to S, \vartheta\big((x,y)\big) = yx^{-1} $.
		Moreover, $\mathrm{Cay}(G,S)$ is clearly bounded since all vertex degrees are equal to the cardinality of $S$. Since $G$ is generated by elements in $S$, $\mathrm{Cay}(G,S)$ is connected. It is easy to see from the definitions, that the path metric coincides with the function $d_S(g,h) = \big| gh^{-1} \big|_S$, where 
			\[
		|g|_S := \min \big\{ L \in \N:\, \mbox{ there exist } s_i \in S, \mbox{ such that } g= s_1s_2 \cdots s_L \big\}
		\]
		if $g \neq e$ and $|e|_S = 0$. 
		Clearly, $G$ acts by isometries on $(G,d_S)$ via right translations. Given $R > 0$ and $g \in G$, we write $B_S(g,R)$ for the (closed) ball of radius $R$ around $x$.

  \subsection{The shift action and $G$-processes} 
Let $G$ be a finitely generated group and $(\mathcal{X},d)$ be a compact metric space. The group $G$ acts continuously on the product space $\cX^G$ (endowed with the product topology) by (right)-translations, 
\[h. (\omega_g)_{g\in G}= (\omega_{gh})_{g\in G},\]
and we call this the {\em shift action}.

A Borel probability measure $\mu \in \mathrm{Prob}(\mathcal{X}^G)$ is called  {\em $G$-invariant} (or just {\em invariant} if the reference to the group is clear) if $\mu(g.B) = \mu(B)$ for every Borel set $B \subseteq \mathcal{X}^G$ and every $g \in G$, where $g.B = \{g.b:\, b \in B\}$. We write $\mathrm{Prob}(\mathcal{X}^G,G)$ for the subset of $G$-invariant Borel probability measures in $\mathrm{Prob}(\mathcal{X}^G)$. A measure $\mu \in \mathrm{Prob}(\mathcal{X}^G,G)$ is {\em ergodic} if $\mu(B) \in \{0,1\}$ for every invariant measurable set $B$, where $B$ is invariant if $\mu (g.B \, \triangle \, B) = 0$ for all $g \in G$. Following \cite{Aus16}, if $\mu\in\mathrm{Prob}(\mathcal{X}^G,G)$ the tuple $(G,\mathcal{X},\mu)$ will be called a $G${\em -process}. If in addition, $\mu$ is ergodic, then we will call say that the process is {\em ergodic}.

  \subsection{Operators} \label{secprelim:operators}
Given a countable discrete set $V$, we consider the standard  $\ell^2$-space
\[
\ell^2(V) = \big\{ f:V \to \C:\, \sum_{v \in V} |f(v)|^2 < \infty \big\},
\]
carrying the scalar product
\[
\langle f,\,k \rangle := \sum_{v \in V} \overline{f(v)}k(v), \quad f,k \in \ell^2(V),
\]
where the bar notation stands for complex conjugation. 
This gives rise to a complex Hilbert space with norm $\|f\| = \sqrt{\langle f,\,f \rangle}$ for $f \in \ell^2(V)$. For $v \in V$, we denote $\delta_v$ for the unique function in $\ell^2(V)$ such that $\delta_v(v) = 1$ and $\delta_v(w) = 0$ if $w \neq v$. A {\em (bounded) operator} is a linear map $H:\ell^2(V) \to \ell^2(V)$ such that $\|H\| := \sup_{0 \neq f \in \ell^2(V)} \|Hf\| / \|f\|$ is finite. The number
$\|H\|$ is called the {\em operator norm} of $H$. For an operator $H$, we write 
\[H(x,y):=\langle \delta_x,\,  H\delta_y \rangle \in \mathbb{C}
\]
for $x,y\in V$. We say that $H$ is {\em self-adjoint} if $H(x,y) = \overline{H(y,x)}$ for all $x,y \in V$. 

 \subsection{Spectral calculus} \label{secprelim:calculus}
 Recall that every bounded and self-adjoint operator $H:\ell^2(V) \to \ell^2(V)$ has a  spectral calculus, cf. \cite{reed1981functional}*{Chapter VII}. There is a unique
projection-valued measure $E_H$, called {\em the spectral measure for $H$}, mapping each Borel set $B \subseteq \R$ to an orthogonal projection $E_H(B)$ onto $\ell^2(V)$ corresponding to the part of the spectrum $\Sigma(H)$ of $H$ contained in $B$. In particular, for each pair $v,w \in V$, a complex Borel measure $\tau_{v,w}$ on $\R$, supported on $\Sigma(H)$, is defined by $\tau_{v,w}(B) = \langle \delta_v,\, E_H(B) \delta_w \rangle$, and one has 
\[
H(v,w) = \langle \delta_v,\, H\delta_w \rangle = \int_{\R} \lambda\, d\tau_{v,w}(\lambda). 
\] 
Moreover, for each real-valued Borel function $f:\R \to \R$ which is bounded when restricted on $\Sigma(H)$, there is a unique self-adjoint operator $f(H):\ell^2(V) \to \ell^2(V)$ such that
\[
\langle \delta_v,\, f(H)\delta_w \rangle = \int_{\R} f(\lambda)\, d\tau_{v,w}(\lambda)
\] 
for all $v,w \in V $, and one has 
\[
\big\| f(H) \big\| = \big\| f\rvert_{\Sigma(H)} \big\|_{\infty}.
\]
We write $\R[x]$ for the ring of polynomials with real coefficients. For $p \in \R[x]$, the degree $\mathrm{deg}(p)$ is the maximal power of a monomial appearing in $p$ with non-zero coefficient. Considering $p = \sum_{i=0}^K a_i x^i$ as a function $\R \to \R$, we have
\[
p(H) = \sum_{i=0}^K a_i H^i, 
\]
which is consistent with the functional calculus, since for a monomial $x \mapsto x^k$, we have 
\[
\langle \delta_v,\, H^k\delta_w \rangle =  \int_{\R} \lambda^k\, d\tau_{v,w}(\lambda)
\]
for all $v,w \in V$.

\section{Sofic groups and local and empirical convergence} \label{sec:soficconvergence}

We define and shortly discuss sofic groups in Section~\ref{sec:soficgroups} and introduce Austin's notion of local and empirical convergence in Section~\ref{sec:LEconv}.

\subsection{Sofic groups and approximations} \label{sec:soficgroups}
Consider a finitely generated group $G$, along with a concrete finite and symmetric generating set $S \subseteq G$, and with its Cayley graph $\mathrm{Cay}(G,S)$ labeled by $ \vartheta((x,y)) = yx^{-1} \in S$. Given a bounded graph $(V, E)$ with edges labeled by elements from $S$, $v \in V$ and $R > 0$ we say that $v$ is an {\em $R$-good vertex (in $(V, E)$)} if there is a graph isomorphism $\varphi:B_{E}(v,R) \to B_S(e,R)$ with $\varphi(v)=e$. We will now introduce the class of sofic groups which allow for approximation via finite graphs that locally look like $\mathrm{Cay}(G,S)$.

\medskip

Here and throughout the paper we will write $|A|$ for the cardinality of a finite set $A$. The following definition is from Weiss \cite{weiss}.
\begin{defn}[Sofic groups]\label{def-sofic}
Let $G$ be a finitely generated group with finite symmetric generating set $S$. Then $G$ is called \emph{sofic}, if for every $\varepsilon>0$ and $R > 0$ there exists a finite graph $(V,E)$ with edge labels from $S$ such that $|V_0| \geq (1-\varepsilon)|V|$ for the set $V_0 \subseteq V$ of $R$-good vertices in $(V,E)$. 
\end{defn}

It can be seen \cite{weiss} that soficity is a property of the group $G$ and is independent of the choice of the concrete finite generating set. 

\begin{defn}\label{def-sofic-approx}
A \emph{sofic approximation} to $G$ is a sequence $(V_n,\sigma_n)$, where $V_n$ is a finite set and $\sigma_n\colon G\to\sym(V_n)$ for each $n\geq 1$, such that
\begin{itemize}
\item For all $g,h \in G$, one has 
\[
\lim_{n \to \infty} \frac{\big| \big\{v \in V_n: \sigma_n^g \big( \sigma_n^h(v) \big) = \sigma_n^{gh}(v) \big\} \big|}{|V_n|} =1, 
\]
\item For all $g \in G \setminus \{e\}$, one has
\[
\lim_{n \to \infty} \frac{\big| \big\{v \in V_n: \sigma_n^g(v) \neq v  \big\} \big|}{|V_n|} =1.
\] 
\end{itemize}
\end{defn}

\medskip
The first property says that  the portion of vertices where $\sigma_n$ restricted to finite subsets of $G$ acts as a homomorphism tends to $1$ as $n \to \infty$.  The second property guarantees that the portion of vertices where the action of $\sigma_n$ restricted to finite subsets of $G$ is free tends to $1$ as $n \to \infty$. 
It can be readily checked that a  finitely generated group is sofic if and only if it has a sofic approximation. Indeed, a sofic approximation gives rise to a sequence of finite graphs $(V_n,E_n)$ which for large $n$ 
satisfy the conditions of Definition \ref{def-sofic} by defining the edge sets $E_n$ as follows
\[
E_n:=\{(v,\sigma_n^s(v))\;|\; v\in V_n,\;s\in S\}.
\]
and an edge labeling as $\vartheta_n:E_n \to S, \vartheta(v,\sigma_n^s(v)) = s$ if $v$ is a $1$-good vertex, and $\vartheta_n(v,\sigma_n^s(v)) = s_0$ otherwise, where $s_0 \in S$ is arbitrary.
Now it is not hard to see that for every $\varepsilon$ and every $R > 0$, there is some $N \in \N$ such that for all $n \geq N$, the portion of $R$-good vertices in $(V_n, E_n)$ is at least $1 - \varepsilon$. Note also that for all such graphs, its vertex degree is bounded by the cardinality of the set $S$. 
We emphasize that while soficity of the underlying group is independent of the choice of generating set $S$, the graphs $(V_n, E_n)$ do depend on $S$. Throughout this work, we will always consider sofic groups $G$ with a fixed finite and symmetric generating set $S$ such that the graphs $(V_n, E_n)$ asymptotically have the same local structure as $\mathrm{Cay}(G,S)$. Emphasizing the graph structure where necessary we will also speak of sofic approximations to $\mathrm{Cay}(G,S)$.
Note further that apriori, the graphs $(V_n, E_n)$ do not need to be edge symmetric, and they may contain loops (i.e.\@ edges of the form $(w,w)$). However, these phenomena can only occur around vertices which are not $R$-good, whose portion will asymptotically be $0$. Thus, modifying the graphs if necessary, there is no loss in generality in assuming always that all graphs from a sofic approximation are indeed edge symmetric and do not contain any loops. 
For more details and background on sofic groups, we refer to \cite{silberstein}*{Chapter~7}. 

\medskip

Recall that a group $G$ is {\em residually finite} if there is a sequence $(\Gamma_n)$ of finite-index, normal subgroups of $G$ such that $\Gamma_{n+1} \leq \Gamma_n$ for all $n \in \N$ and $\bigcap_{n=1}^{\infty} \Gamma_n = \{e\}$.

\begin{defn}
    Let $G$ be a residually finite group. A sequence $(V_n,\sigma_n)$ is called a {\em residually finite approximation of} $G$ if there exists a sequence $(\Gamma_n)$ of finite-index normal subgroups of $G$ with $\Gamma_{n+1}\leq \Gamma_n$ such that $V_n=G/\Gamma_n$ and $\sigma_n\colon G\to \mathrm{Sym}(V_n)$ is defined as $\sigma_n^g(v\Gamma_n)=gv\Gamma_n$ for every $n\geq 1$, $g\in G$ and $v\in V_n$.
\end{defn}

It is readily checked that every residually finite approximation of $G$ is a sofic approximation of $G$.

\subsection{Local and empirical convergence}  \label{sec:LEconv}
Let $(\cX,d)$ be a compact metrizable space. Given a countable set $V$, we write $\mathrm{Prob}(\cX^V)$ for the space of all Borel probability measures on the (metrizable) product topological space $\cX^V$.

\medskip

We now introduce the notions of locally weak$^*$ convergence and local and empirical convergence for measures on $\mathcal{X}^{V_n}$ for sofic approximations $(V_n, \sigma_n)$, as defined by Austin in \cite{Aus16}. To this end, let $G$ be a sofic finitely generated group and $(V_n,\sigma_n)$ a sofic approximation of $G$.  Fixing $V=V_n$ and $\sigma = \sigma_n$, for $\rho=(\rho_v)_{v\in V}\in \cX^V$, set 
\[ \Pi_v^\sigma(\rho):=(\rho_{\sigma^g(v)})_{g\in G}\in \cX^G.\]
Each $\rho \in \cX^{V}$ has an associated probability measure on $\cX^G$, called its \emph{empirical distribution},
\[P_\rho^\sigma:=\frac{1}{|V|}\sum_{v\in V}\delta_{\Pi_v^\sigma(\rho)}.\]

Fix $\mu \in \mathrm{Prob}(\mathcal{X}^G,G)$. For an open set $\cO\subset \prob(\cX^G)$, set
\[\Omega(\cO,\sigma):=\{\rho\in \cX^V\mid P_\rho^\sigma \in \cO\}.\]

\medskip

\begin{defn}[lw$^{*}$-convergence and le-convergence]
	Let $(V_n,\sigma_n)$ be a sofic approximation of $G$, $\mu \in \mathrm{Prob}(\mathcal{X}^G,G)$ and $\mu_n\in \prob(\cX^{V_n})$ for every $n\geq 1$. 
    \begin{enumerate}
        \item[$\bullet$] The sequence $(\mu_n)$ \emph{locally weak$^*$ converges} (or \emph{lw$^{*}$-converges} for short) to $\mu$, denoted by $\mu_n\overset{lw*}{\longrightarrow}\mu$, if for every weak$^*$-neighbourhood $\cO$ of $\mu$ we have 
        \[
        \lim_{n \to \infty} \frac{\big| \big\{ v \in V_n:\, \big{(} \Pi^{\sigma_n}_v \big{)}_{*} \mu_n \in \mathcal{O} \big\} \big|}{|V_n|} = 1.
        \] 
        \item[$\bullet$] The sequence $(\mu_n)$ \emph{locally and empirically converges} (or \emph{le-converges} for short)  to $\mu$, denoted by $\mu_n\overset{le}{\longrightarrow}\mu$, if $\mu_n\overset{lw*}{\longrightarrow}\mu$ and for every weak$^*$-neighbourhood $\cO$ of $\mu$ we have
        \[\lim_{n \to \infty} \mu_n(\Omega(\cO,\sigma_n)) = 1.\]
    \end{enumerate} 
 
\end{defn}

\begin{rem}
    Note that $\mu_n\overset{lw*}{\longrightarrow}\mu$ is equivalent to the fact that for all continuous $f:\cX^G \to \C$ and all $\varepsilon>0$ we have that
\[\frac{1}{|V_n|} \left|\left\{v\in V_n:\, \Bigg| \int_{\cX^{V_n}} f((\Pi_v^{\sigma_n})(\rho))d\mu_n(\rho)-\int_{\cX^G}f(\omega)d\mu(\omega) \Bigg| <\varepsilon\right\} \right| \xrightarrow{n\to \infty} 1.\]
The second property in the definition of le-convergence says that
\[\mu_n\left( \left\{\rho\in \cX^{V_n}\,: \, \Bigg| \frac{1}{|V_n|}\sum_{v\in V_n} f(\Pi_v^{\sigma_n}(\rho))-\int_{\cX^G}f(\omega)d\mu(\omega) \Bigg| <\varepsilon \right\} \right)\xrightarrow{n\to \infty} 1.\]
\end{rem}

\begin{defn}[Sofic and residually finite models]
	Let $G$ be a sofic group, $\Sigma=(V_n,\sigma_n)$ a sofic approximation for $G$ and $\mathcal{X}$ a compact metrizable space. Whenever $(G,\mathcal{X},\mu)$ is a $G$-process and there exists a sequence $(\mu_n)_{n\geq 1}$ with $\mu_n\in \mathrm{Prob}(\mathcal{X}^{V_n})$ such that $\mu_n\stackrel{lw*}{\longrightarrow}\mu$ (resp. $\mu_n\stackrel{le}{\longrightarrow}\mu$), we say that $(V_n,\sigma_n,\mu_n)$ is a {\em sofic model} (resp. {\em ergodically sofic model}) for $(G,\mathcal{X},\mu)$. If moreover $(V_n,\sigma_n)$ is a residually finite approximation and each $\mu_n$ is $V_n$-invariant, we call $(V_n,\sigma_n,\mu_n)$ a {\em residually finite model} (resp. {\em ergodically residually finite model}) for $(G,\mathcal{X},\mu)$.
\end{defn}

\medskip

The following assertion is due to Austin.

\begin{prop}[{\cite{Aus16}*{Corollary~5.7}}] \label{prop:leergodic}
Let $\mu \in \mathrm{Prob}(\mathcal{X}^G,G)$ and $(\mu_n)$ be a sequence with $\mu_n \in \mathrm{Prob}(\mathcal{X}^{V_n})$ for a sofic approximation $(V_n,\sigma_n)$. If $\mu$ is ergodic, then 
\[\mu_n\stackrel{lw*}{\longrightarrow} \mu \;\;\;\Longrightarrow\;\;\; \mu_n\stackrel{le}{\longrightarrow} \mu.\]
\end{prop}

\begin{rem} 
 The above proposition states that in the case of ergodic processes, it is sufficient to find $lw^*$-convergent sequences in order to obtain examples for le-convergence. 
\end{rem}

\section{Periodic measures and local empirical convergence  on sofic groups} \label{sec:PA}

This section is devoted to the construction of measures that locally and empirically converge to a prescribed invariant probability measure, making use of weak-$*$ approximation by periodic measures. The general framework is established in Section~\ref{sec:periodicmeasures}, see Proposition~\ref{prop:periodic_locally_approx}, and investigated in depth for periodically approximable groups in Section~\ref{sec:PAsub}, see Proposition~\ref{prop:EPA_quotient}.

\subsection{Approximation by periodic measures} \label{sec:periodicmeasures}
Let $G$ be a sofic group, and $(\mathcal{X},d)$ be a compact metric space. The next lemma shows that the sets of $G$-processes admitting sofic, ergodically sofic, residually finite or ergodically residually finite models are weak-$*$ closed.

\begin{lem}\label{lemma:stability_under_limits}
	Let $(\mu_n)$ be a sequence in $\mathrm{Prob}(\mathcal{X}^G,G)$ and suppose that $\mu_n\overset{w*}{\to}\mu \in \mathrm{Prob}(\mathcal{X}^G,G)$.
	\begin{enumerate}
		\item[\textup{(i)}] If each $(G,\mathcal{X},\mu_n)$ admits a sofic model (resp. ergodically sofic model), then $(G,\mathcal{X},\mu)$ admits a sofic model (resp. ergodically sofic model).
		\item[\textup{(ii)}] If each $(G,\mathcal{X},\mu_n)$ admits a residually finite model (resp. ergodically residually finite model), then $(G,\mathcal{X},\mu)$ admits a residually finite model (resp. ergodically residually finite model).
	\end{enumerate}
\end{lem}

\begin{proof}
	This follows from straightforward diagonal sequence arguments.
\end{proof}

 \begin{defn}[Periodic measure]
A measure $\mu \in \mathrm{Prob}(\mathcal{X}^G,G)$ is {\em periodic} if $\mu$ has finite support, or equivalently if $\supp(\mu)=\bigcup_{i=1}^kG.\omega_i$, where each $\omega_i\in \mathcal{X}^G$ is such that $\mathrm{stab}_G(\omega_i):=\{g\in G:g.\omega_i=\omega_i\}$ is a finite-index subgroup.
 \end{defn}

 By invariance, one immediately sees that for each periodic $\mu$,  there are $k \in \N$ along with $\omega_i \in \mathcal{X}^G$ and $a_i > 0$ for $1 \leq i \leq k$, such that $\sum_{i=1}^k a_i =1$ and 
	\[
	\mu =  \sum_{i=1}^k a_i \frac{1}{|G.\omega_i|} \sum_{\zeta \in G.\omega_i} \delta_\zeta. 
	\]

\begin{prop}\label{prop:periodic_locally_approx}
	Let $G$ be a sofic group and let $\mu\in \mathrm{Prob}(\mathcal{X}^G,G)$ be a periodic measure. 
 \begin{enumerate}
     \item[\textup{(i)}] There is a sofic model $(V_n,\sigma_n,\mu_n)$ for $(G,\mathcal{X},\mu)$ such that $(\Pi_{v}^{\sigma_n})_{*}\mu_n = \mu$ for all $n$ and $v \in V_n$. 

     \item[\textup{(ii)}] If $G$ is residually finite, then the sofic model from \textup{(i)} can be chosen to be a residually finite model.
 \end{enumerate}
\end{prop}

For the proof of the proposition we will also need the following basic lemma.

\begin{lem}\label{lemma:fundamental_domain}
	Let $N\leq H\leq G$. Choose sets $C$ and $D$ of representatives of left cosets of $H$ in $G$ and of $N$ in $H$, respectively. Then,
	\begin{enumerate}
		\item[\textup{(i)}] $CD$ contains exactly one element from each coset of $N$ in $G$,
		\item[\textup{(ii)}] $CD=\bigsqcup_{c\in C}cD=\bigsqcup_{d\in D}Cd$, and in particular $|CD|=|C|\cdot|D|$.
	\end{enumerate}
\end{lem}

\begin{proof}
	First, note that if $g\in G$, then $g= ch$ for some $c\in C$ and $h\in H$. Moreover, there exist $d\in D$ and $t\in N$ such that $h=dt$, so $g=cdt\in cdN$. Now, if $c,c'\in C$ and $d,d'\in D$ are such that $cdN=c'd'N$, since $dN$ and $d'N$ are subsets of $H$ we have that $cH\cap c'H\neq \varnothing$, which yields $c=c'$. This directly implies that $d=d'$, so $cd=c'd'$. Thus, we have proved (i).
	
	To see that (ii) is true, suppose $cD\cap c'D\neq \emptyset$. We can then find $d,d'\in D$ with $cd=c'd'$, so that $cdN=c'd'N$ and, again, $c=c'$. The other decomposition is analogous.
\end{proof}

\begin{proof}[Proof of Proposition~\ref{prop:periodic_locally_approx}]
Fix an arbitrary sofic approximation $\hat{\Sigma}=(\hat{V}_n,\hat{\sigma}_n)$, and let $\Omega\subseteq \mathcal{X}^G$ be finite such that $\mathrm{supp}(\mu)=\bigsqcup_{\omega\in \Omega}G.\omega$. The subgroup $\bigcap_{\omega\in \Omega}\mathrm{stab}_G(\omega)\leq G$ is a finite intersection of finite-index subgroups of $G$, so it is of finite index as well. Let $N\leq \bigcap_{\omega\in \Omega}\mathrm{stab}_G(\omega)$ be a normal subgroup of $G$ of finite index.

	We proceed to prove (i). Define, for each $n\geq 1$, the map $\sigma_n\colon G\to \mathrm{Sym}(\hat{V}_n\times G/N)$ given by
	$$\sigma_n^g(v,hN)=(\hat{\sigma}_n^{g}(v),ghN),$$
	for $g,h\in G$ and $v\in \hat{V}_n$. Notice that if $g\in G\setminus N$, then $gNhN\neq hN$, so $\sigma_n^g(v,hN)\neq (v,hN)$ for all $h\in G$. If on the other hand $g\in N\setminus\{e\}$, then
	\begin{align*}
		\lim_{n\to\infty}\frac{1}{|\hat{V}_n\times G/N|}|\{(v,hN)\in \hat{V}_n\times G/N:\sigma_n^g(v,hN)\neq(v,hN)\}|&\\
		=\lim_{n\to\infty}\frac{|G/N|}{|\hat{V}_n||G/N|}|\{v\in \hat{V}_n:\hat{\sigma}_n^{g}(v)\neq v\}|=1&.
	\end{align*}
	Hence, $\Sigma:=(\hat{V}_n\times G/N,\sigma_n)$ is asymptotically free. Moreover,
	$$\sigma_n^{gg'}(v,hN)=[\sigma_n^g\circ \sigma_n^{g'}](v,hN)\iff \hat{\sigma}_n^{gg'}(v)=[\hat{\sigma}_n^{g}\circ \hat{\sigma}_n^{g'}](v),$$
	for all $g,g',h\in G$ and $v\in \hat{V}_n$, from where it follows that 
	\begin{align*}
		\lim_{n\to\infty}\frac{|\{(v,hN)\in \hat{V}_n\times G/N:\sigma_n^{gg'}(v,hN)=[\sigma_n^g\circ \sigma_n^{g'}](v,hN)\}|}{|\hat{V}_n\times G/N|}&\\
		=\lim_{n\to\infty}\frac{|G/N|}{|\hat{V}_n||G/N|}|\{v\in \hat{V}_n:\hat{\sigma}_n^{gg'}(v)=\hat{\sigma}_n^{g}\circ \hat{\sigma}_n^{g'}(v)\}|=1&.
	\end{align*}
	Thus, $\Sigma$ is a sofic approximation.

    Now, if for each $\omega\in \Omega$ there is $\mu_n^\omega\in \mathrm{Prob}(\mathcal{X}^{\hat{V}_n\times G/N})$ such that $\mu_n^\omega$ locally weak-$*$ converges to the uniform measure on
    $G.\omega$, it is clear that the measure $\sum_{\omega\in\Omega}\mu(G.\omega)\mu_n^\omega$ will locally weak-$*$ converge to $\mu$. We can thus assume that $\Omega=\{\omega\}$ and $\mu$ is the uniform measure on $G.\omega$.
 
	Choose a set $C\subseteq G$ of representatives of the left cosets of $\mathrm{stab}_G(\omega)$ in $G$, and let $D\subseteq \mathrm{stab}_G(\omega)$ be a set of representatives of the left cosets of $N$ in $\mathrm{stab}_G(\omega)$. Then $CD$ is a fundamental domain for $N$ in $G$ by Lemma \ref{lemma:fundamental_domain}. Fix now $n\geq 1$. Given $t\in CD$, let $\rho^{t}_n\in \mathcal{X}^{\hat{V}_n\times G/N}$ be the configuration satisfying $\rho^{t}_n(v,hN)= \omega(ht)$ for all $h\in G$ and all $v\in \hat{V}_n$. This is well-defined, since if $hN=h'N$ then $t^{-1}h^{-1}h't\in N\subseteq \mathrm{stab}_G(\omega)$, so $\omega(ht)=(ht.\omega)(e)=(h't.\omega)(e)=\omega(h't)$. Further, for $t\in CD$, $g,h\in G$ and $v\in \hat{V}_n$,
	$$\left[\Pi^{\sigma_n}_{(v,hN)}(\rho^{t}_n)\right](g)=\rho^{t}_n(\sigma_n^g(v,hN))=\rho^{t}_n(\hat{\sigma}_n^{g}(v),ghN)=\omega(ght)=(ht.\omega)(g),$$
	so $\Pi^{\sigma_n}_{(v,hN)}(\rho^{t}_n)=ht.\omega$. Since $D\subseteq \mathrm{stab}_G(\omega)$, we have $cd.\omega=cd'.\omega$ for all $c\in C$, $d,d'\in D$, so $\rho^{cd}_n=\rho^{cd'}_n$ and it follows that
	$$\mu_n:=\frac{1}{|C D|}\sum_{t\in C D}\delta_{\rho^{t}_n}=\frac{1}{|C|}\sum_{c\in C}\delta_{\rho^{c}_n},$$
	where we assume without loss of generality that $e\in D$. With this, we have for all $h\in G$ and $v\in \hat{V}_n$ that
	\[\left[\Pi^{\sigma_n}_{(v,hN)}\right]_*\mu_n=\frac{1}{|C|}\sum_{c\in C}\delta_{\Pi^{\sigma_n}_{(v,hN)}(\rho^{c}_n)}=\frac{1}{|C|}\sum_{c\in C}\delta_{hc.\omega}=\frac{1}{|G.\omega|}\sum_{\zeta\in G.\omega}\delta_{\zeta}=\mu.\]
	This implies the assertion~(i), as desired.
 
 Now, we prove (ii). If $G$ is residually finite, we can take $\hat{\Sigma}=(V_k,\hat{\sigma}_k)$ to be a residually finite approximation, so that $\hat{V}_n=G/\operatorname{ker}(\hat{\sigma}_n)$, and define $N_n:= N\cap \operatorname{ker}(\hat{\sigma}_n)$ for every $n\geq 1$. Letting $\sigma_n$ to be the homomorphism $G\to\operatorname{Sym}(G/N_n)$ associated to the quotient map $\pi_n\colon G\to G/N_n$, it is immediate that the sequence $(G/N_n,\sigma_n)$ is a residually finite approximation. Again, now that we have fixed the sofic approximation it suffices to consider the case where $\Omega=\{\omega\}$ and $\mu$ is the uniform measure on $G.\omega$.

 Fix $n\geq 1$. Define $\rho_n\in \mathcal{X}^{G/N_n}$ by $\rho_n(gN_n)=\omega(g)$ (this is well-defined, since $N_n\leq N\leq \mathrm{stab}_G(\omega)$). Note that for all $v \in G/N_n$, the map $\Pi_v^{\sigma_n}$ defines a bijection $(G/N_n).\rho_n\rightarrow G.\omega$. Indeed, if $g\in G$, a preimage of $g.\omega$ under $\Pi_v^{\sigma_n}$ is $(h_v^{-1}gN_n).\rho_n$, where $h_v\in G$ is such that $\pi_n(h_v)=v$. Hence, the map is surjective. Moreover, for $t,g,h\in G$, we have
$$\Pi_{gN_n}^{\sigma_n}(hN_n.\rho_n)(t)=(hN_n.\rho_n)(tgN_n)=\rho_n(tghN_n)=\omega(tgh)=(gh.\omega)(t),$$
	so $\Pi_{gN_n}^{\sigma_n}(hN_n.\rho_n)=gh.\omega$ for each $\omega\in\Omega$. With this, if $gh_1.\omega=gh_2.\omega$ and $w_i:=h_iN_n$ for $i=1,2$, then $h_1.\omega=h_2.\omega$, which in turn implies, for all $t\in G$,
	$$(w_1.\rho_n)(tN_n)=\rho_n(th_1N_n)=\omega(th_1)=\omega(th_2)=\rho_n(th_2N_n)=(w_2.\rho_n)(tN_n).$$
	Thus, $w_1.\rho^\omega_n=w_2.\rho_n$, so that $\Pi_{gN_n}^{\sigma_n}$ defines an injective map. We conclude that this map is bijective, as claimed.
	
	Consider $\mu_n\in \mathrm{Prob}(\mathcal{X}^{G/N_n},G/N_n)$ to be the uniform measure supported over the orbit $(G/N).\rho_n$.
	Then $(\Pi_v^{\sigma_n})_*\mu_n$ is the uniform measure upon $G.\omega$, since for any $g\in G$ we have
	$$\mu_n((\Pi_v^{\sigma_n})^{-1}(\{g.\omega\}))=\frac{1}{|(G/N_n).\rho_n|}=\frac{1}{|G.\omega|}.$$
	This yields (ii).
\end{proof}

\subsection{Local and empirical convergence for PA groups or quotients} \label{sec:PAsub}

	Sequences that are le-conver\-gent can be found for every ergodic action of a PA group, as will be demonstrated in this section.
	
	\medskip

	\begin{defn}[PA groups and EPA groups]
		A countable group $G$ is  {\em periodically approximable} or is said to have {\em property PA} (and $G$ is called a {\em PA group} for short) if for every finite set $\mathcal{A}$, the set of periodic $G$-invariant probability measures is weak-$*$-dense in $\mathrm{Prob}(\mathcal{A}^G,G)$. \\ The group $G$ is {\em ergodically periodically approximable} or is said to have {\em property EPA} (and $G$ is called an {\em EPA group} for short) if for every finite set $\mathcal{A}$, the set of ergodic periodic $G$-invariant probability measures is weak-$*$-dense in $\mathrm{Prob}(\mathcal{A}^G,G)$.  
	\end{defn} 


As mentioned above, the property PA was  introduced by  Bowen in \cite{Bow03}, where it was proved that finitely generated free groups have this property. In \cite{Kec12}, Kechris introduced the property MD, which is equivalent to the property PA (see \cite{BK20}*{p. 2698}). 
It is known that the following groups have the property PA: free groups $\mathbb{F}_r$ \cite{Kec12}*{p. 486}, amenable residually finite groups \cite{ren2018}, free products of groups that are either finite and non-trivial or have property PA \cite{tucker2015weak}*{Theorem 4.8}, subgroups of groups with property PA, groups containing a finite-index subgroup with the property PA \cite{Kec12}*{p. 486}, groups of the form $H\rtimes \mathbb{F}_r$ with $H$ amenable and residually finite \cite{BK20}*{p. 2699}. On the other hand, examples of groups without property 
PA include for instance all non-residually finite groups \cite{Bow03}*{Section~7.2.}, as well as $\mathrm{SL}_n(\Z)$ for $n>2$  \cite{Kec12}*{p. 466}, and $\mathbb{F}_2\times \mathbb{F}_2$ \cite{BK20}*{p. 2699}.

\medskip

We denote by $R(G)$ the {\em residual subgroup of} $G$, that is, the intersection of all finite-index normal subgroups of $G$. Given a group homomorphism $\pi\colon G\to Q$, define the map $\Phi^{\pi}\colon \mathcal{X}^Q\to \mathcal{X}^G$ by $\Phi^{\pi}(\omega)= \omega \circ \pi$ for $\omega \in \mathcal{X}^{Q}$.

\begin{prop}\label{prop:EPA_quotient}
    Let $\mathcal{A}$ be a finite set, and let $\mathcal{P}(\mathcal{A}^G)$ and $\mathcal{P}_e(\mathcal{A}^G)$ be the sets of periodic and periodic ergodic measures in $\mathrm{Prob}(\mathcal{A}^G,G)$, respectively. Let $\pi\colon G\to Q$ be a surjective group homomorphism. 
    \begin{enumerate}
        \item[\textup{(i)}] If $Q$ has property PA, then $\Phi^{\pi}_*(\mathrm{Prob}(\mathcal{A}^{Q},Q))\subseteq\overline{\mathcal{P}(\mathcal{A}^G)}$. In particular, any $G$-process $(G,\mathcal{A},\mu)$ with $\mu\in\Phi^{\pi}_*(\mathrm{Prob}(\mathcal{A}^{Q},Q))$ admits a sofic model.
        \item[\textup{(ii)}] If $Q$ has property EPA, then $\Phi^{\pi}_*(\mathrm{Prob}(\mathcal{A}^{Q},Q))\subseteq\overline{\mathcal{P}_e(\mathcal{A}^G)}$. In particular, every $G$-process $(G,\mathcal{A},\mu)$ with $\mu\in\Phi^{\pi}_*(\mathrm{Prob}(\mathcal{A}^{Q},Q))$ admits an ergodically sofic model.
    \end{enumerate}
    The inclusions from (i) and (ii) are equalities if moreover $Q$ is isomorphic to $G/R(G)$. 
\end{prop}

The following standard lemma will be needed for the proof of the proposition. 

\begin{lem}\label{lemma:push-forward}
    For a group homomorphism $\pi\colon G\to Q$, we have the following.
    \begin{enumerate}
        \item[\textup{(i)}] The map $\Phi^\pi$ is  continuous, and $\Phi^\pi(\pi(g).\omega)=g.\Phi^\pi(\omega)$ for all $\omega\in \mathcal{X}^Q$, $g\in G$.
        \item[\textup{(ii)}] The push-forward map $\Phi^\pi_*\colon \mathrm{Prob}(\mathcal{X}^Q)\to \mathrm{Prob}(\mathcal{X}^G)$ is weak-* continuous and sends \\ 
        $Q$-invariant measures to $G$-invariant measures.
         \item[\textup{(iii)}] If $\pi$ is additionally surjective, then the maps $\Phi^\pi$ and $\Phi^\pi_*$ are injective. 
        \end{enumerate}
\end{lem}

\begin{proof}
  This is standard and left for the reader to prove. 
\end{proof}

\begin{proof}[Proof of Proposition~\ref{prop:EPA_quotient}]
    We first show the inclusions from (i) and (ii). Let $\mu\in\mathrm{Prob}(\mathcal{A}^{Q},Q)$. Since $Q$ is a   PA group, there is a sequence $(\mu_n)_n$ in $\mathrm{Prob}(\mathcal{A}^{Q},Q)$ of periodic measures weak-* converging to $\mu$, so by continuity of $\Phi^{\pi}_*$ we get $\Phi^{\pi}_*\mu_n\to\Phi^{\pi}_*\mu$.  We need to check that every $\Phi^{\pi}_*\mu_n$ is periodic.  
 
 If $\omega\in \mathrm{supp}(\mu_n)$ and $U$ is an open neighborhood of $\Phi^{\pi}(\omega)$, then $\omega\in (\Phi^{\pi})^{-1}(U)$, which implies $\Phi^{\pi}_*\mu_n(U)=\mu_n((\Phi^{\pi})^{-1}(U))>0$, so $\Phi^{\pi}(\omega)\in \mathrm{supp}(\Phi^{\pi}_*\mu_n)$. This shows that $\Phi^{\pi}(\mathrm{supp}(\mu_n))$ is a subset of $\mathrm{supp}(\Phi^{\pi}_*\mu_n)$. Conversely, if $y\in \mathrm{supp}(\Phi^{\pi}_*\mu_n)$, then $\mu_n((\Phi^{\pi})^{-1}B(y,1/k))>0$ for every $k >1$. Since $\mathrm{supp}(\mu_n)$ is finite, there must exist some $\omega\in \mathrm{supp}(\mu_n)$ such that $\omega\in (\Phi^{\pi})^{-1}(B(y,1/k))$ for infinitely many $k$, which yields $\Phi^{\pi}(\omega)=y$ and hence $\mathrm{supp}(\Phi^{\pi}_*\mu_n)=\Phi^{\pi}(\mathrm{supp}(\mu_n))$. We can now write $\mathrm{supp}(\mu_n)=\bigsqcup_{i=1}^kQ.\omega_i$ for some elements $\omega_i\in \mathcal{A}^Q$ with $|Q.\omega_i|<\infty$. We obtain
 $$\mathrm{supp}(\Phi^{\pi}_*\mu_n)=\Phi^\pi\left(\bigsqcup_{i=1}^kQ.\omega_i\right)=\bigsqcup_{i=1}^kG.\Phi^{\pi}(\omega_i),$$
 where we use the fact that $\Phi^{\pi}(\pi(g).\omega_i)=g\cdot\Phi^{\pi}(\omega_i)$ from part (i) of Lemma \ref{lemma:push-forward}. Note that the second equality above is justified by surjectivity of $\pi$ in combination with part~(iii) of Lemma~\ref{lemma:push-forward}. 
 We conclude that $\Phi^{\pi}_*\mu_n$ is periodic. Note that if $k=1$ (so that each $\mu_n$ is ergodic), then $\mu$ is ergodic as well. This proves the inclusions from (i) and (ii). 
 
 Both ``in particular'' statements follow directly from Proposition \ref{prop:periodic_locally_approx} together with Lemma \ref{lemma:stability_under_limits}. Hence (i) and (ii) are proven.

 Now, we prove the converse inclusions under the additional assumption that $Q\simeq G/{R}(G)$. For this, it suffices to show that $\mathcal{P}(\mathcal{A}^G)$ is a subset of $\Phi^{\pi}_*(\mathrm{Prob}(\mathcal{A}^{Q},Q))$ when $G$ has property PA, since $\Phi^{\pi}_*(\mathrm{Prob}(\mathcal{A}^{Q},Q))$ is closed by compactness and continuity of $\Phi^{\pi}_*$. Let $\mu\in \mathrm{Prob}(\mathcal{A}^G,G)$ be a periodic measure, and write $\supp(\mu)=\bigsqcup_{i=1}^kG.\omega_i$ for some elements $\omega_i\in \mathcal{A}^G$ with $|G.\omega_i|<\infty$. Let $H=\bigcap_{i=1}^k\mathrm{stab}_G(\omega_i)$. Since $H$ has finite index in $G$, it contains a normal, finite index subgroup $N$ of $G$, and  we have ${R}(G)\leq N$, so the canonical group homomorphism $\gamma\colon Q\to G/N$ satisfies $\gamma\circ \pi=\pi_N$, where $\pi_N\colon G\to G/N$ is the quotient map. This translates into $\Phi^\pi\circ \Phi^\gamma=\Phi^{\pi_N}$.

 Define $\tilde{\omega}_i\in \mathcal{A}^{G/N}$ by $\tilde{\omega}_i(\pi_N(g))=\omega_i(g)$. This is well defined, since if $g^{-1}g'\in N \leq H=\bigcap_{i=1}^k\mathrm{stab}_G(\omega)$ then $\omega_i(g)=(g^{-1}g'.\omega_i)(g)=\omega_i(g')$. Now set 
$$\tilde{\mu}=\sum_{i=1}^k\frac{\mu(G.\omega_i)}{|G.\omega_i|}\sum_{\zeta \in (G/N).\tilde{\omega}_i}\delta_{{\zeta}}\in \mathrm{Prob}(\mathcal{A}^{G/N},G/N).$$
Indeed, $\tilde{\mu}$ is a probability measure. To see this, note first that for all $g,g'\in G$, 
 $$\Phi^{\pi_N}(\pi_N(g).\tilde{\omega}_i)(g')=\tilde{\omega}_i(\pi_N(g'g))=(g.(\tilde{\omega}_i\circ \pi_N))(g')=(g.\Phi^{\pi_N}(\tilde{\omega}_i))(g'),$$
 so $\Phi^{\pi_N}(\pi_N(g).\tilde{\omega}_i)=g.\Phi^{\pi_N}(\tilde{\omega}_i)=g. \omega_i$ for each $i$. Hence $\Phi^{\pi_N}$ defines a sur\-jec\-tive map $(G/N).\tilde{\omega}_i$ $\to G.\omega_i$. Moreover, since $\pi_N$ is 
 sur\-jec\-tive, $\Phi^{\pi_N}$ is injective by part~(iii) of Lemma \ref{lemma:push-forward}, so $|(G/N).\tilde{\omega}_i|=|G.\omega_i|$. Hence $\tilde{\mu}$ is a convex sum of the uniform measures upon the orbits $(G/N).\tilde{\omega}_i$. 
 
 We have also proven that $\pi_N(g).\tilde{\omega}_i$ is the only preimage of $g\cdot \omega_i$ under $\Phi^{\pi_N}$. It follows that for every $g\in G$, 
 \begin{align*}
     \Phi^{\pi_N}_*\tilde{\mu}(\{g.\omega_i\})&=\tilde{\mu}((\Phi^{\pi_N})^{-1}(\{g.\omega_i\}))=\tilde{\mu}(\{\pi_N(g).\tilde{\omega}_i\})=\frac{\mu(G.\omega_i)}{|G.\omega_i|}=\mu(\{g.\omega_i\}).
 \end{align*}
 Therefore $\Phi^\pi_*(\Phi^{\gamma}_*\tilde{\mu})=\Phi^{\pi_N}_*\tilde{\mu}=\mu$, 
 where by part~(ii) of the previous lemma, $\Phi^{\gamma}_*\tilde{\mu}$ is $Q$-invariant. 
 Consequently, $\mu\in \Phi^\pi_*(\mathrm{Prob}(\mathcal{A}^Q,Q))$.
\end{proof}

\begin{cor} \label{cor:PAEPA}
	Let $G$ be a PA (resp.\@ EPA) group, and let $\mathcal{A}$ be a finite set. Then every $G$-process $(G,\mathcal{A},\mu)$ admits a residually finite model (resp. ergodically residually finite model).
\end{cor}

\begin{rem}
Shriver has obained in  \cite{Shr23}*{Proposition~4.1}  a characterization of PA groups through weak-$*$ approximation of measures on $\mathcal{A}^G$ via periodic measures of the form 
\[
\frac{1}{|V_n|} \sum_{v \in V_n} \delta_{\Pi^{\sigma_n}_v(\rho)},
\]
with $(V_n,\sigma_n)$ being a sofic approximation consisting of homomorphisms and $\rho \in \mathcal{A}^{V_n}$. While this construction is tailor-made for the empirical distributions to resemble the original measure, the difficulty remains to emulate the periodic measures in model space and also to obtain residually finite models. We have developed the arguments needed for this purpose in the proof of Proposition~\ref{prop:periodic_locally_approx}.
\end{rem}

\begin{proof}[Proof of Corollary~\ref{cor:PAEPA}]
This is a direct consequence of Proposition \ref{prop:EPA_quotient} (applied for $\pi=\mathrm{id}\colon G\to G$). That the models can taken to be residually finite models follows by noting that $\Phi^{\text{id}}_*$ is the identity map, and by applying Proposition \ref{prop:periodic_locally_approx} and Lemma \ref{lemma:stability_under_limits}.
\end{proof}

\begin{rem}
    The group $G/{R}(G)$ is called the largest residually finite quotient of $G$ because every other residually finite quotient $T$ of $G$ factors as a quotient of $G/{R}(G)$. Now, one could be lead to define a {\em largest PA quotient of} $G$ to be a PA quotient $Q$ such that any other PA quotient $T$ factors as a quotient of $Q$, but it turns out that such an object might not exist.  As an example, take $G=\mathbb{F}_2\times \mathbb{F}_2$ and assume $N\trianglelefteq \mathbb{F}_2\times \mathbb{F}_2$ is such that $Q:=(\mathbb{F}_2\times \mathbb{F}_2)/N$ is a largest PA quotient for $\mathbb{F}_2\times \mathbb{F}_2$. Since $\mathbb{F}_2\times 1$ and $1\times \mathbb{F}_2$ are both PA quotients of $\mathbb{F}_2 \times\mathbb{F}_2$, they have to factor as quotients of $Q$, implying $N\leq (1\times \mathbb{F}_2)\cap (\mathbb{F}_2\times 1)=1$, so $Q=\mathbb{F}_2\times \mathbb{F}_2$. However, $\mathbb{F}_2\times \mathbb{F}_2$ is not a PA group (see \cite{BK20}*{p. 2699} combined with \cite{ji2021}). 
\end{rem}

\begin{exmp}[Baumslag-Solitar groups] \label{exa:BSgroups}
For $m,n\geq 1$ the Baumslag-Solitar group with parameters $m,n$ is defined as
$$\mathrm{BS}(m,n)=\langle a,b\mid ab^m=b^na\rangle.$$
As an HNN extension of $\Z$, $\mathrm{BS}(m,n)$ is a sofic group \cite{collins2010free}.
The groups $\mathrm{BS}(1,n)$ are residually finite and amenable, hence they are EPA by \cite{ren2018} and so they fall under the hypotheses of Corollary \ref{cor:PAEPA}. If $m,n>1$, then the group $\mathrm{BS}(m,n)$ is a non-ascending HNN extension of $\Z$, so it is a non-amenable group, and if moreover the sets of prime divisors of $m$ and $n$ differ, it is well-known that $\mathrm{BS}(m,n)$ is non-Hopfian \cite{baumslag1962some}, so it cannot be residually finite. In particular, if $m,n>1$ are coprime integers then the group $\mathrm{BS}(m,n)$ is a non-amenable, non-Hopfian sofic group. Nevertheless, as proven in \cite{kelley2020subgroup}*{Lemma 3.2} the largest residually finite quotient of $\mathrm{BS}(m,n)$ is $Q:=\Z\left[\tfrac{1}{mn}\right]\!\rtimes \Z$, where the generator of $\Z$ acts by multiplication by $m/n$. Note that $Q$ is amenable as an abelian-by-abelian extension. It is also finitely generated, and since it can be written as $Q=H_1H_2$ for two abelian subgroups $H_1,H_2\leq Q$, it is metabelian by \cite{ito1955produkt}. It now follows from \cite{hall1959finiteness}*{Theorem 1} that $Q$ is residually finite, so it has the EPA property. We conclude from Proposition~\ref{prop:EPA_quotient} that, denoting by $\pi$ the quotient map $\mathrm{BS}(m,n)\to Q$, every measure in $\Phi^\pi_*(\mathrm{Prob}(\mathcal{A}^Q,Q))$ gives rise to a $G$-process admitting an ergodically sofic model.

\end{exmp}

\section{Continuous operators and weak spectral convergence} \label{sec:weakconvergence}

In this section we introduce the relevant families of operators which we call {\em continuous operators}. Existence and the properties of their underlying 
finite volume-analogs via sofic approximations, called {\em induced approximations} below are discussed in detail in Section~\ref{sec:inducedapprox}. Afterwards, we define the density of the states measure and prove weak convergence along induced approximations, cf.\@ Theorem~\ref{lemma-weakconvergence} in Section~\ref{sec:weakconv}.

\medskip

Given a graph $(V,E)$, we
 say that $M \in \N$ is a {\em finite hopping range parameter} for an operator $H:\ell^2(V) \to \ell^2(V)$ if $H(x,y) = 0$ whenever $d_E(x,y)>M $.
We now consider a finitely generated sofic group $G$ with symmetric generating set $S$.  Let $(V,E) = \mathrm{Cay}(G,S)$ be the Cayley graph of $G$ with respect to $S$ and canonical edge labeling $\vartheta((x,y)) = yx^{-1}$.
 Moreover, for a compact metric space $\mathcal{X}$, a family $\{H^{\omega}\}_{\omega \in \mathcal{X}^G}$ of operators $H^{\omega}:\ell^2(G) \to \ell^2(G)$ is said to be {\em equivariant} if 
	\[
	H^{\omega}\big(xg,yg\big) = H^{g.\omega}(x,y), \quad \mbox{ for all } \quad x,y,g \in G, \quad \omega \in \mathcal{X}^G.
	\]  
Defining the unitary operators $U_g:\ell^2(G) \to \ell^2(G)$, $U_g u(x) = u(xg)$, the above equivariance condition is equivalent to
\[
H^{g.\omega} = U_{g} H^{\omega} U_{g^{-1}}, \quad \omega \in \mathcal{X}^G, \quad g \in G.
\]

\begin{defn}[Continuous operator]
	A family $\{H^{\omega}\}_{\omega \in \mathcal{X}^G}$ of operators $H^{\omega}:\ell^2(G) \to \ell^2(G)$ is called a {\em continuous operator over $\mathrm{Cay}(G,S)$ (with parameter $M$)} if 
	\begin{itemize}
		\item for all $x,y \in G$, the map $\omega \mapsto H^{\omega}(x,y)$ is continuous;
		\item the operator $H^{\omega}$ is self-adjoint for each $\omega \in \mathcal{X}^G$;
		\item there is $M \in \N$ which is a finite hopping range parameter for all $H^{\omega}$, $\omega \in \mathcal{X}^G$;
		\item the family $\{H^{\omega}\}$ is equivariant. 
	\end{itemize}
	We say that the continuous operator $\{H^{\omega}\}$ is of {\em bounded interaction} if in addition, for all $\omega_1, \omega_2 \in \mathcal{X}^G$ and each $x \in G$, we have 
	\[
	\omega_{1}\rvert_{B_S(e,M)} = \omega_{2}\rvert_{B_S(e,M)} \quad \Longrightarrow \quad H^{\omega_1}(e,x) = H^{\omega_2}(e,x).
	\]
\end{defn}

It is clear from the definition that the choice of generating set $S$ matters. However, we will assume that $S$ is fixed as soon as $G$ is. Thus, with a slight abuse of terminology, we will also say that $\{H^{\omega}\}$ is a {\em continuous operator over $G$} or just that $\{H^{\omega}\}$ is a {\em continuous operator} if the reference to $G$ and $S$ is clear. Given a process $(G,\mathcal{X},\mu)$ we will also refer to $\{H^{\omega}\}$ as a {\em random operator} (with law $\mu$), see the terminology used in the title of this work and in the introduction.

\begin{rem} \label{rem:boundedinteraction}
    We make a couple of further remarks. 
    \begin{itemize}
        \item One could introduce an extra parameter for the bounded interaction condition which might be different from $M$. However, one can clearly match the two parameters by taking their maximum.
          \item Note continuity of the maps $\omega \mapsto H^{\omega}(x,y)$ together with the finite hopping range property and equivariance  automatically imply that 
          that the operators are uniformly bounded in operator norm, i.e.\@ 
		\[
		\sup_{\omega \in \mathcal{X}^G} \big\| H^{\omega}\big\| < \infty, 
		\]
        see also the formula~\eqref{eqn:bounded!} below. 
        \item  In the case that $\mathcal{X} = \mathcal{A}$ is finite, bounded interaction implies   that the functions $\omega \mapsto H^{\omega}(x,y)$ are locally constant, and hence automatically continuous. 
       \item  
       If the coefficients $H^{\omega}(x,y)$ are only taken from a finite set of possible values, the continuity property also forces the maps $\omega \mapsto H^{\omega}(x,y)$ to be locally constant. Increasing the parameter $M$ if necessary, we can make sure that $\{H^{\omega}\}$ is of bounded interaction in this case. 
    \end{itemize}
\end{rem}

\subsection{Induced approximations for sofic Cayley graphs} \label{sec:inducedapprox}
We now define finite-dimensional approximants for continuous operators over Cayley graphs $\mathrm{Cay}(G,S)$ of finitely generated sofic groups $G$. To this end, fix a sofic approximation $(V_n,\sigma_n)$ to $G$ consisting of $S$-labeled graphs.

\begin{defn}[Induced approximations]\label{def-randomapprox} 
	Let $M \in \N$. 
	A sequence $\big( \{ H^{\rho}_n \}_{\rho \in \mathcal{X}^{V_n}} \big)_{n \in \N}$ of families of  operators $H_n^{\rho}:\ell^2(V_n) \to \ell^2(V_n)$ over $(V_n,E_n)$ is a {\em weak approximation over $(V_n,\sigma_n)$}  for a continuous operator $\{H^{\omega}\}$ over $G$ with parameter $M$ if 
	\begin{itemize}
		\item for all $n \in \N$ and $x,y \in V_n$, the map $\rho \mapsto H_n^{\rho}(x,y)$ is Borel,
  \item for all $n \in \N$ and $\rho \in \mathcal{X}^{V_n}$, the operator $H_n^{\rho}$ is self-adjoint,
  \item $M$ is a finite hopping range parameter for all operators $ H^{\rho}_n$, 
			\item for all $n \in \N$ and all $4 M$-good vertices $w \in V_n$, we have 
		\begin{align*}
			H_n^{\rho}(w,v)=H^{\prod_w^{\sigma_n}(\rho)}(e,g)
		\end{align*}
		for all $\rho \in \mathcal{X}^{V_n}$, $v \in B_{E_n}(w,M)$ and the unique $g \in B_S(e,M)$ with $\sigma_n^g(w)=v$. 
	\end{itemize}
 We say that $\big( \{ H^{\rho}_n \}_{\rho \in \mathcal{X}^{V_n}} \big)_{n \in \N}$ is an {\em induced approximation} for $\{H^{\omega}\}$  if it is a weak approximation for $\{H^{\omega}\}$ and in addition, 
 \begin{itemize}
 \item  
 for each $n \in \N$ and $\rho \in \mathcal{X}^{V_n}$, and all $v,w \in V_n$, we have $$H_n^{\rho}(v,w) \in \{H^{\omega}(x,y):\, x,y \in V,\, \omega \in \mathcal{X}^G \} \cup \{0\}.
		$$
 \end{itemize}
\end{defn}

\begin{rem} \label{rem:bounded}
	The fifth property says that all $H_n^{\rho}$ do not take coefficients other than coefficients from $\{H^{\omega}\}$ or $0$. The fourth property is a stronger condition for $4M$-good vertices. Here, it is required that the $H_n^{\rho}$ take their coefficients from $H^{\omega}$ at places where the  geometric and color structure of $(V_n,E_n)$ and $\mathrm{Cay}(G,S)$ are identical. Note further that 
	all graphs $(V_n,E_n)$ and $\mathrm{Cay}(G,S)$ have a uniform bound on their vertex degree. Since the uniform boundedness of $\{H^{\omega}\}$ is equivalent to
	\begin{align} \label{eqn:bounded!}
	\sup_{x,y \in G} \sup_{\omega \in \mathcal{X}^G} \big| H^{\omega}(x,y) \big| < \infty,
	\end{align}
	the fifth property guarantees that 
	 all operator norms are uniformly bounded, i.e.\@ 
	\begin{align*} 
		\max\big\{ \sup_{n \in \N} \sup_{\rho \in \mathcal{X}^{V_n}} \big\| H_n^{\rho} \big\|,\, \sup_{\omega \in \mathcal{X}^G} \|H^{\omega}\| \big\} < \infty. 
	\end{align*}
\end{rem}

\begin{exmp} \label{exa:Schrodinger}
	Let $(V,E) = \mathrm{Cay}(G,S)$ be the Cayley graph over a sofic group generated by a finite, symmetric set $S \subseteq G$. 
	Let $\Delta$ be the  graph Laplacian $\Delta\colon \ell^2(G)\to \ell^2(G)$, 
	\[\Delta(u(x))=\sum_{y\sim x} (u(y)- u(x))\]
	where the sum runs over all vertices connected with an edge to $x$. Then the collection $\{H^{\omega}\}$, with 
	$H^\omega\colon\ell^2(G)\to \ell^2(G)$ given by
	\[H_F^\omega u(x)=\Delta u(x) + F(\omega(x))u(x)\]
	for  $\omega\in \mathcal{X}^G$, $u\in \ell^2(G)$, $x\in G$ and a countinuous map $F\colon \mathcal{X}\to \R$, defines a continuous operator with parameter $M=1$. The family $\{H^{\omega}_F\}$ consists of discrete Schr\"odinger operators with real-valued potentials $F \circ \omega$. The equivariance follows from a short computation for $g,h \in G$ and $u \in \ell^2(G)$:
	\begin{align*}
		& (U_{g} H^\omega_F U_{g^{-1}} u)(h) = (H^\omega_F U_{g^{-1}} u)(hg) = (\Delta U_{g^{-1}} u)(hg)+ F(\omega(hg))(U_{g^{-1}} u)(hg) \\
		& \quad =\sum_{z:z\sim hg} \big( (U_{g^{-1}}u)(z)-(U_{g^{-1}}u)(hg) \big) +F(\omega(hg))u(h)\\
		& \quad =\sum_{z:zg^{-1} \sim h} \big( u(zg^{-1})-u(h) \big) +F(\omega(hg))u(h) \\
		& \quad =\sum_{z:z\sim h} \big( u(z)-u(h) \big) +F(g.\omega(h))u(h) = H^{g.\omega}_Fu(h).
	\end{align*}
	Given a sofic approximation $(V_n, \sigma_n)$, we can define an  induced approximation 
 $(\{H_n^{\rho}\}_{\rho})_{n \in \N}$ by setting
	\[H_n^\rho u(v)=\Delta u(v)+F(\rho(v))u(v)\]
	for $\rho\in \mathcal{X}^{V_n}$, $u\in \ell
	^2(V_n)$, and $v\in V_n$. Given a process $(G,\mathcal{X},\mu)$ we call $\{H_F^{\omega}\}$ a {\em random Schr\"odinger operator} with {\em potential $F$}. 
\end{exmp}

\medskip

		The next proposition and its proof are based on the following general fact: if $(V,E,\sigma)$ is a finite graph in a sofic approximation, and if $v \in V$ is a $Q$-good vertex for some $Q \in \N$, then for each $w \in B_{E_n}(v,Q)$, there is a unique $g \in G$ with $|g|_S \leq Q$ such that $\sigma^g(v)=w$.

	\begin{prop}[Existence of induced approximations]  \label{prop:goodapprox}
		Suppose $(V,E) = \mathrm{Cay}(G,S)$ for a sofic group $G$ generated by a finite set $S$, and let $(V_n,E_n,\sigma_n)$ be a sofic approximation. 
		Let  $\{H^{\omega}\}_{\omega \in \mathcal{X}^G}$ be a continuous operator with parameter $M$, where $\mathcal{X}$ is a compact metric space.
		Suppose that in addition, at least one of the following holds:
		\begin{itemize}
			\item all maps $\sigma_n: G \to \mathrm{Sym}(V_n)$ are homomorphisms;
			\item the continuous operator $\{H^{\omega}\}$ is of bounded interaction.
		\end{itemize}  
	Then an  induced approximation $\big( \{H_n^{\rho}\} \big)$ can be defined by setting
		\begin{align} \label{eqn:2M!}
			H^{\rho}_n(w,v) = H^{\Pi_w^{\sigma_n}(\rho)}(e,g),\quad \rho \in \mathcal{X}^{V_n}, \quad v,w \in V_n,
		\end{align}
		if both $v,w $ are $2M$-good vertices, $v \in B_{E_n}(w,M)$ and $g \in G$ is unique with $|g|_S \leq M$ and $\sigma^g(w) =v$, and setting $H^{\rho}(v,w) = 0$ in all other cases. 
	\end{prop}

 	\begin{rem}
        In the respective contexts of approximation, analogous constructions have appeared before in the literature, see e.g.\@ 
		Section~3.1 in \cite{Ele08},
		 Lemma~2.2 and  the discussion thereafter in \cite{SS15}, or Lemma~9.8. and the preceding discussion in \cite{PogPhD}. 
	\end{rem}

	\begin{proof}
		Let $\{H_n^{\rho}\}$ be defined as above. Note that the operators are well-defined since for an $M$-good vertex $v$ there is a unique $g \in G$ with $|g|_S \leq M$ and $\sigma^g(w) =v$. The measurability of the maps 
		\[
		\rho \mapsto H_n^{\rho}(w,v) = H^{\Pi_w^{\sigma_n}(\rho)}(e,g)
		\]
		is an immediate consequence from the measurability of the maps
		\[
		\rho \mapsto \Pi_w^{\sigma_n}(\rho), \quad \omega \mapsto H^{\omega}(e,g),
		\]
		and the measurability of the constant $0$ map. 
		
		\medskip
        By construction, $M$ is a finite hopping range parameter for all $H_n^{\rho}$, and the operators $H_n^{\rho}$ take values only from operators in $\{H^{\omega}\}$ or $0$. 
		
\medskip
It remains to be shown that the operators $H_n^{\rho}$ are all self-adjoint. Here we will need one of the additional assumptions. 
First note that we can assume that $v,w \in V_n$ are both $2M$-good and $d_{E_n}(v,w) \leq M$, since otherwise we get from the definition $H_n^{\rho}(v,w) = 0 = H_n^{\rho}(w,v)$. Let $\bar{g}\in B_S(e,M)$ the unique element with $\sigma_n^{\bar{g}}(w)=v$. Since also $v$ is $2M$-good, $\bar{g}^{-1}$ is the unique element such that $\sigma_n^{\bar{g}^{-1}}(v) = w$. Using the equivariance condition, we find 
		\begin{align*}
			H_n^{\rho}(w,v) &= H^{\prod_w^{\sigma_n}(\rho)}(e,\bar{g}) = H^{\overline{g}\prod_w^{\sigma_n}(\rho)}(\bar{g}^{-1},e). 
		\end{align*} 
		Further, by self-adjointness, we have 
		\[
		H^{\overline{g}\prod_w^{\sigma_n}(\rho)}(\bar{g}^{-1}, e) = \overline{ H^{\overline{g}\prod_w^{\sigma_n}(\rho)}(e,\bar{g}^{-1} )}.
		\]
		Now assume first that $\sigma_n: g \mapsto \sigma_n^g$ is a homomorphism. 
		Then
		\begin{align*}
		\bar{g} \Pi_w^{\sigma_n}(\rho) &= \bar{g} \big( \rho\big( \sigma^h_n(w)  \big) \big)_{h \in G} = \big( \rho\big( \sigma^{h\bar{g}}_n(w)  \big) \big)_{h \in G} \\
		&= \big( \rho\big( \sigma^{h}_n(\sigma_n^{\bar{g}}(w))  \big) \big)_{h \in G} =   \Pi_{v}^{\sigma_n}(\rho). 
		\end{align*} 
		With what we showed above, we arrive at
		\begin{align} \label{eqn:self-adjoint!}
		H_n^{\rho}(w,v) = \overline{ H^{\overline{g}\prod_w^{\sigma_n}(\rho)}(e,\bar{g}^{-1})} = \overline{ H^{\prod_v^{\sigma_n}(\rho)}(e,\bar{g}^{-1} )} = \overline{H_n^{\rho}(v,w)}
		\end{align}
		in this case. As for the second case, note that even if we don't know whether $\sigma_n$ is an isomorphism, we know by the almost homomorpism property of sofic approximations, together with the facts that $|\bar{g}|_S \leq M$ and $v,w$ are $2M$-good that for each $h \in B_S(e,M)$
		\begin{align*}
				 \bar{g}\Pi_w^{\sigma_n}(\rho ) (h) &=  \rho\Big( \sigma_n^{h\bar{g}}(w) \Big)  =   \rho\Big( \sigma_n^{h} (\sigma_n^{\bar{g}}(w)) \Big) = \rho\big( \sigma_n^{h}(v) \big)  = \Pi^{\sigma_n}_v(\rho)(h),
			\end{align*}
		and thus, 
		\[
		\bar{g}\Pi_w^{\sigma_n}(\rho)\rvert_{B_S(e,M)} =  \Pi_v^{\sigma_n}(\rho)\rvert_{B_S(e,M)}. 
		\]
		With the additional assumption that $\{H^{\omega}\}$ is of bounded interaction, the same calculation as in~\eqref{eqn:self-adjoint!} concludes the proof. 
		\end{proof}
	
		We will now consider powers of measurable operators. Note that for a  bounded operator $H:\ell^2(V) \to \ell^2(V)$ over a bounded graph $(V,E)$ and $k \in \N$ with finite hopping range parameter $M \in \N$, the coefficients of the operator $H^k$ are given by the formula
		\begin{align} \label{eqn:powers}
			H^{k}(x,y) = \sum_{(z_i)} \prod_{i=1}^k H(z_i, z_{i+1}),
		\end{align}
		where the sum runs over all tupes $(z_i) \in V^{k+1}$  connecting $x$ and $y$ (i.e.\@ $z_1 = x$ and $z_{k+1}=y$) with the additional property that $d_{E}(z_i, z_{i+1}) \leq M$. In particular, $H^k$ is self-adjoint whenever $H$ is self-adjoint.
	
		\medskip
		With these observations, one quickly arrives at the following for a continuous operator over a Cayley graph $(V,E)= \mathrm{Cay}(G,S)$ with sofic approximation $(V_n,E_n,\sigma_n)$. 
		
		\begin{lem} \label{lem:powers}
			Let $k \in \N$. 
			If $\{H^{\omega}\}$ is a continuous operator with parameter $M \in \N$, then so is the family $\{(H^{\omega})^k\}$ with parameter $k\cdot M$.
		\end{lem}
	
		The next lemma shows that $k$-powers of  weak approximations for $\{H^{\omega}\}$ give rise to weak approximations for $\{(H^{\omega})^k\}$.

	\begin{lem} \label{prop:inducedpowers}
		Let $\big(\{H_n^{\rho}\}_{\rho \in \mathcal{X}^{V_n}}\big)$ be a weak approximation  for a continuous operator $\{H^{\omega}\}_{\omega \in \mathcal{X}^G}$ with parameter $M \in \N$. Assume in addition that at least one of the following holds:
				\begin{itemize}
				\item all maps $\sigma_n: G \to \mathrm{Sym}(V_n)$ are homomorphisms;
				\item the continuous operator $\{H^{\omega}\}$ is of bounded interaction.
			\end{itemize}
		For each $k \in \N$, $\big( \{ (H_n^{\rho})^k \}_{\rho \in \mathcal{X}^{V_n}} \big)$ is a weak approximation  for $\{(H^{\omega})^k\}_{\omega \in \mathcal{X}^G}$ (with parameter $k\cdot M$). 
	\end{lem}

	\begin{proof}
We will give the proof for the fourth condition of a weak approximation. All other requirements are easily verified via the formula~\eqref{eqn:powers}.

\medskip

Let $k \in \N$. We need to show that for $n \in \N$, $\rho \in \mathcal{X}^{V_n}$,
		\begin{align} \label{eqn:power!}
			(H_n^{\rho})^k \big( v, \sigma_n^g(v) \big) = \big( H^{\Pi_v^{\sigma_n}(\rho)} \big)^k(e,g) 
		\end{align}
		for every $4kM$-good vertex $v \in V_n$ and $g \in B_S(e,k\cdot M)$. We proceed by induction on $k$. The case $k=1$ follows from the definition of an induced approximation. 
	 
		 Now assume that the claim holds for $k-1$, $k \geq 2$. We fix some $4kM$-good vertex $v \in V_n$ (if it exists, for otherwise there is nothing to be shown). To relax the notation we will write $H_n$ for $H_n^{\rho}$ and $H$ for $H^{\Pi^{\sigma_n}_v(\rho)}$. For $g \in B_S(e,k  M)$ we have
		\begin{eqnarray*}
			H_n^k\big(v, \sigma_n^g(v)\big)
			&=& \sum_{z \in V_n} H_n^{k-1}\big( v,\, z \big) \cdot  H_n(z,\, \sigma_n^g(v)).
		\end{eqnarray*} 
	We will show that 
		\begin{eqnarray*} 
			H_n^k\big(v,\, \sigma_n^g(v) \big) &=& \sum_{x \in G}  H^{k-1}(e,x) \cdot H(x,g)  = H^k(e,g).
		\end{eqnarray*}
		To verify the latter equality, we observe the following.
		The condition 
		$$H_n^{k-1}\big( v,\, z \big) \cdot  H_n(z,\, \sigma_n^g(v)) \neq 0$$ 
		forces 
		\[
		z \in B_{E_n}(\sigma_n^g(v),M) \cap B_{E_n}(v,(k-1)M).
		\] 
		These $z$ are in a one-to-one correspondence with some unique $x \in B_S(g,M) \cap B_S(e,(k-1)M)$ such that $\sigma_n^{x}(v) = z$ and $\sigma_n^{gx^{-1}}\big(\sigma_n^x(v) \big) = \sigma_n^g(v)$. 
		Similarly as in the proof of existence of induced approximations, we note that 
		\[
		x\Pi_v^{\sigma_n}(\rho)(h) = \rho \big( \sigma_n^{hx}(v) \big) =  \rho\big( \sigma_n^h (\sigma_n^x(v))  \big) = \Pi^{\sigma_n}_{\sigma_n^x(v)}(\rho)(h)
		\]
		for all $h \in G$ if $\sigma_n$ is a homomorphism, and if $H^{\omega}$ is of bounded interaction, then it is enough to note that the above equality holds at least for $h \in B_S(e,M)$.  
		Consequently, by equivariance, and since $z= \sigma_n^x(v)$ is $4M$-good,
		\begin{align*}
			H_n(\sigma_n^x(v), \sigma_n^g(v)) &= H^{\Pi^{\sigma_n}_{\sigma^x_n(v)}(\rho)}(e, gx^{-1}) = H^{x\Pi^{\sigma_n}_{v}(\rho)}(e, gx^{-1}) \\
			&= H^{\Pi_{v}^{\sigma_n}(\rho)}(x,g) = H(x,g),
		\end{align*}
		and by induction hypothesis, $H_n^{k-1}(v,\sigma_n^x(v)) = H^{k-1}(e,x)$. If $z \notin B_{E_n}(\sigma_n^g(v), M)$, then there is no $h \in B_S(e,M)$ such that $z = \sigma_n^{hg}(v) = \sigma_n^h (\sigma_n^g(v)) $, and consequently, 
		\[
		H_n(z,\sigma_n^g(v)) = 0 = H(x,g)
		\]
		for all $x \notin B_S(g,M)$. If $z \notin B_{E_n}(v, (k-1)M)$, then there is no $t \in B_S(e, (k-1)M)$ such that $\sigma_n^t(v) = z$. Since both $H_n^{k-1}$ and $H^{k-1}$ have hopping range $(k-1)M$, we obtain $H_n^{k-1}(v,z) = 0 = H^{k-1}(e,t)$ for $t \notin B_S(e,4(k-1)M)$.  This finishes the proof of equality~\eqref{eqn:power!}.
	\end{proof}
	
\medskip

We will now make use of the spectral calculus introduced in Section \ref{secprelim:calculus}. 
Note that for a polynomial $p \in \mathbb{R}[x]$, it is an immediate consequence from Lemma~\ref{lem:powers} that for an induced approximation $\big(\{H_n^{\rho}\}\big)$  for $\{H^{\omega}\}$ with parameter $M$ as above, 
$\big(\{p(H_n^{\rho})\} \big)$ is a  weak approximation for $\{p(H^{\omega})\}$ with parameter $\mathrm{deg}(p)\cdot M$.
This leads to the first step in weak convergence along polynomials.

\begin{lem}\label{lemma-induced} 
	Let $\big\{ H_{\omega} \big\}_{\omega \in \mathcal{X}}$ be a continuous operator over a Cayley graph $(V,E) = \mathrm{Cay}(G,S)$, where $\mathcal{X}$ is a compact metric space. Let $(V_n,\sigma_n)$ be a sofic approximation. Assume that one of the following holds
	\begin{itemize}
			\item all maps $\sigma_n: G \to \mathrm{Sym}(V_n)$ are homomorphisms;
	\item the continuous operator $\{H^{\omega}\}$ is of bounded interaction.
	\end{itemize}
	Then for every  induced approximation $\big(\{H_n^{\rho}\}_{\rho \in \mathcal{X}^{V_n}}\big)_n$  for $\{H^\omega\}$ and for all $p \in \R[x]$, we have
\[
\lim_{n \to \infty} \sup_{\rho \in \mathcal{X}^{V_n}} \frac{\big| v \in V_n:\, p(H_n^{\rho})(v,v) =  p(H^{\Pi_v^{\sigma_n}(\rho)})(e,e)  \big|}{|V_n|} = 1.
\]
\end{lem}

\begin{proof}
	Suppose that $\{H^{\omega}\}$ has parameter $M$. Fix $p \in \R[x]$. 
	It is immediate from Lemma~\ref{lem:powers} that $\{p(H^{\omega})\}$ is a continuous operator with parameter $\deg(p)M$. Consequently, $\big(\{p(H_n^{\rho})\}\big)$ is a weak approximation for $\{p(H^{\omega})\}$ by 	Lemma~\ref{prop:inducedpowers}, and we clearly have boundedness of  $p(H_n^{\rho})$ uniformly both in $n$ and $\rho$. 
	Hence, $$p(H_n^{\rho})(v,v) = p(H^{\Pi_v^{\sigma_n}(\rho)})(e,e)$$ 
	for every $4\mathrm{deg}(p)M$-good vertex $v \in V_n$ and for each $\rho \in \mathcal{X}^{V_n}$. Since $(V_n,E_n,\sigma_n)$ is a sofic approximation, the ratio of $4\mathrm{deg}(p)M$-good vertices in $V_n$ converges to $1$ as $n \to \infty$. This concludes the proof. 
\end{proof}

In the following we write  $C(\R)$ for the space of  continuous real-valued functions on $\R$.

\begin{cor}\label{cor-av-eq-id}
With the assumptions of the previous lemma, for all $f \in C(\R)$, we have  
\[
\lim_{n \to \infty} \sup_{\rho \in \mathcal{X}^{V_n}} \left|
\frac{1}{|V_n|}\sum_{v\in V_n}\Big(\langle \delta_v,f(H_n^\rho)\delta_v\rangle-\langle \delta_{e}, f(H^{\Pi_v^{\sigma_n}(\rho)})\delta_{e}\rangle\Big) \right| = 0.\]
\end{cor}
\begin{proof}
	The uniform bound in operator norm of $\big(\{H_n^{\rho}\}_{\rho \in \mathcal{X}^{V_n}}\big)_n$ together with the spectral calculus gives $\sup_{n \in \N} \sup_{\rho \in \mathcal{X}^{V_n}} \|f(H_n^{\rho})\| < \infty$. Similarly, one has $\sup_{\omega \in \mathcal{X}^G}\|f(H^{\omega})\| < \infty$. For $f=p$ being a polynomial, the claim now follows from Lemma~\ref{lemma-induced}. For general $f$, the above boundedness properties together with the continuous spectral calculus imply that one can with no loss of generality assume that $f$ is compactly supported. The proof is concluded by Lemma~\ref{lemma-induced} together with the Stone-Weierstra{\ss} theorem on uniform densenss of polynomials in the space of continuous functions on a compact set. 
\end{proof}

\subsection{Weak convergence of the density of states measure} \label{sec:weakconv}
Let $\{H^{\omega}\}_{\omega \in \mathcal{X}^V}$ be a {\em measurable family} of operators over (some edge labeled bounded graph) $(V,E)$, i.e.\@ all  maps 
\[
\omega \mapsto H^{\omega}(v,w), \quad v,w \in V
\]
are Borel measurable. We define a Borel probability measure $\vartheta^{\omega}$ on $\R$  for each $\omega \in \mathcal{X}^V$ as follows, depending if $V$ is finite or infinite. Namely, if $V$ is finite we set 
\begin{align} \label{eqn:tracemeasure}
\vartheta^\omega(B) := \frac{1}{|V_n|} \mathrm{Tr} E_{H^{\omega}(B)} = \frac{1}{|V_n|}\sum_{v\in V_n} \langle\delta_v, E_{H^\omega}(B)\delta_v\rangle,
\end{align}
where $\mathrm{Tr}$ denotes the standard trace for quadratic matrices. 
In the situation where $V$ is infinite, we set 
\[
\vartheta^{\omega}(B) = \langle \delta_e,\, E_{H^{\omega}}(B)\delta_e \rangle.
\]
In both cases, $E_{H^{\omega}}$ denotes the spectral measure for the self-adjoint operator $H^{\omega}$, cf.\@
Section~\ref{secprelim:calculus}.
Note that by boundedness of $H^{\omega}$, $\vartheta^{\omega}$ is supported on the spectrum $\Sigma(H^{\omega})$ of $H^{\omega}$, a compact set. 

\medskip

For a measurable function $f:\R \to \R$ with bounded restrictions on compact sets, we set
\[
\vartheta^{\omega}(f) := \int_{\R} f(s)\, d \vartheta^{\omega}(s).
\]
Note that $C(\R)$ belongs to this class. We denote by $C_b(\R)$ the subspace of all bounded elements in $C(\R)$.

\begin{defn}[Integrated density of states]
	The function 
	\begin{align*}
		N:\R \to [0,1], \quad N(\beta) = \int_{\mathcal{X}^V} \vartheta^{\omega}\big( ]-\infty, \beta] \big)\, d\mu(\omega)
	\end{align*}
	is called the {\em integrated density of states (IDS)} of the underlying measurable family $\{H^{\omega}\}$ with measure $\mu$.  The {\em density of states measure} of $\{H^{\omega}\}$ is given by 
	\[
	N^{\prime}: C_b(\R) \to \R, \quad N^{\prime}(f) = \int_{\mathcal{X}^V} \vartheta^{\omega}(f)\, d\mu(\omega) = \int_{\mathcal{X}^V} \int_{\R} f(s)\, d\vartheta^{\omega}(s) \, d\mu(\omega).
	\]
\end{defn}

 \medskip 
 In mathematical physics, $N(\beta)$ is usually interpreted as an averaged number of quantum states below a fixed energy level $\beta$, per unit value, see the explanations in the introduction. 
There are natural relations to a certain von Neumann algebra carrying a faithful, normal and finite trace. We will explore these connections in Section~\ref{sec:monotoneOP} below.

\begin{rem}
	Note that for a finite graph $(V,E)$ and $\beta \in \N$, the number $|V| \cdot \vartheta^{\omega} \big( ]-\infty,\beta] \big)$ corresponds to the number of eigenvalues of $H^{\omega}$ less or equal than $\beta$, counted with their multiplicities. Further, it follows from the formula~\eqref{eqn:powers} that for every $p \in \R[x]$, the maps $\omega \mapsto p(H^{\omega})(x,y)$ are continuous.
	Combining this with the continuous and measurable calculus, one observes that the aforementioned maps $\omega \mapsto \vartheta^{\omega}(\cdot)$ are measurable. Hence, the above integrals are well-defined.  
\end{rem}

 We are now able to establish weak convergence of the density of states measure. 

 \begin{thm}[Weak convergence of the density of states measure]\label{lemma-weakconvergence}
	Let $(V_n,\sigma_n,\mu_n)$ be an ergodically sofic model for a $G$-process $(G,\mathcal{X}
 ,\mu)$, and $\{H^{\omega}\}_{\omega \in \mathcal{X}^G}$ be a continuous operator over a Cayley graph $\mathrm{Cay}(G,S)$. Suppose that at least one of the following holds:
	\begin{itemize}
		\item all maps $\sigma_n: G \to \mathrm{Sym}(V_n)$ are homomorphisms;
		\item the continuous operator $\{H^{\omega}\}$ is of bounded interaction.
	\end{itemize}
	 Then, for every induced approximation $ (\{H_n^{\rho}\})$  for $\{H^{\omega}\}$ we have: for every $\varepsilon> 0$ and for each $f \in C_b(\R)$, 
\[
\lim_{n \to \infty} \mu_n \left( \left\{ \rho \in \mathcal{X}^{V_n}:\, \big| \vartheta_n^{\rho}(f) - N^{\prime}(f) \big| < \varepsilon  \right\} \right) = 1.
\]
\end{thm}

\begin{proof} 
	Recall that by Remark~\ref{rem:bounded} that 
	\[
	\sup_{n \in \N} \sup_{\rho \in \mathcal{X}^{V_n}} \|H^{\rho}_n\| < \infty, \quad \mbox{ and } \quad \sup_{\omega \in \mathcal{X}^G} \|H^{\omega}\| < \infty.
	\]
	Together with the spectral theory of bounded self-adjoint operators, this yields that there is a compact set $\Sigma \subseteq \R$ containing the supports of the measures $\vartheta^{\rho}_n$ and $\vartheta^{\omega}$. By the Weierstra{\ss} approximation theorem, it is enough to verify the claim for every $p \in \R[x]$, where we identify $p$ with some element in $C_b(\R)$ coinciding with $p$ on $\Sigma$ and being continued constantly outside of $\Sigma$. Via  Lemma~\ref{lem:powers}, it can be seen that the map $\omega \mapsto \langle \delta_e,\, p(H^{\omega})\delta_e \rangle$ is continuous. Local and empirical convergence of $(\mu_n)$ implies that for every $\varepsilon > 0$, the sequence of numbers
\[\mu_n\left(\left\{\rho\in\mathcal{X}^{V_n}\, :\, \Big| \frac{1}{|V_n|}\sum_{v\in V_n} \langle\delta_{e},p(H^{\Pi_v^{\sigma_n}(\rho)})\delta_{e}\rangle-\int_{\mathcal{X}^G} \langle\delta_{e},p(H^\omega)\delta_{e}\rangle d\mu(\omega) \Big| <\epsilon\right\}\right)\]
converges to $1$ as $n \to \infty$. 
Combining this with Corollary \ref{cor-av-eq-id} immediately yields convergence of 
\begin{align*}
	\mu_n \left( \left\{ \rho \in \mathcal{X}^{V_n}\,:\, \Bigg| \frac{1}{|V_n|}\sum_{v\in V_n} \langle\delta_{e},p(H^{\Pi_v^{\sigma_n}(\rho)})\delta_{e}\rangle - N^{\prime}(p) \Bigg| < \varepsilon \right\} \right)
\end{align*}
to $1$, as claimed. 
\end{proof}

\begin{cor}   \label{lemma-weakconvergence-consequence}
	Under the assumptions of the previous theorem, there is a sequence of sets $A_n \subseteq \mathcal{X}^{V_n}$ with $\lim_{n \to \infty} \mu_n(A_n) = 1$ such that for every sequence $(\rho_n)$ with $\rho_n \in A_n$, the  probability measures $\vartheta_{n}^{\rho_{n}}$ converge weakly to $N^{\prime}$, i.e.\@
	\begin{align*}
		\lim_{n \to \infty} \vartheta_{n}^{\rho_{n}}(f) = N^{\prime}(f), \quad f \in C_b(\R). 
	\end{align*}
\end{cor}

\begin{proof}
	Similarly as before, we can assume that the measures $\vartheta^{\rho}_n$ and $N^{\prime}$ are supported in one and the same compact set $\Sigma \subseteq \R$. Further, the space $C(\Sigma)$ of continuous real-valued functions on $\Sigma$ is separable by the Stone-Weierstra{\ss} theorem. Consequently we can fix a countable subset $S=\{p_l:l \in \N\}$ that is dense in $C(\Sigma)$ with respect to the supremum norm. We fix an arbitrary enumeration $\N \to S, i \mapsto p_i$ of $S$.
	
	\medskip
	
	{\bf Claim.} For each $l \in \N$ there is some $N_l \in \N$ such that $\mu_n(A_{n,l}) > 1 - 2^{-l}$ for $n \geq N_l$, where 
	\[
	A_{n,l} := \big\{ \rho \in \mathcal{X}^{V_n}:\, |\vartheta_n^{\rho}(p_i) - N^{\prime}(p_i) | < 1/l   \mbox{ for all } 1 \leq i \leq l \big\}.
	\]
	For the proof of the claim, set for $1 \leq i \leq l$
	\[
	A_{n,l}^{(i)} := \big\{ \rho \in \mathcal{X}^{V_n}:\, |\vartheta_n^{\rho}(p_i) - N^{\prime}(p_i) | < 1/l \big\}.
	\]
	Applying the previous lemma on all $p_i$, $1 \leq i \leq l$, there is some $N_l$ such that $\mu_n(A_{n,l}^{(i)}) > 1 - 2^{-l}/l$ for all $n \geq N_l$ and all $1 \leq i \leq l$. With $A_{n,l} = \bigcap_{i=1}^l A_{n,l}^{(i)}$ we find that 
	\[
	\mu(A_{n,l}) = 1 - \mu\big( \bigcup_{i=1}^l (A_{n,l}^{(i)})^c \big) \geq 1 - l \cdot 2^{-1}/l = 1 - 2^{-l}.
	\]
	We now use the claim. To this end, we can assume with no loss of generality that $N_{l+1} > N_l$ for all $l \in \N$. Now the sequence $(A_n)$ can be defined for $n \geq N_1$ by setting $A_n := A_{n,l}$, where $l \in \N$ is the unique integer such that $n \in [N_l, N_{l+1})$. (For $n < N_1$, we can simply pick any non-empty measurable subset $A_n \subseteq \mathcal{X}^{V_n}$.) Then $$\liminf_{n \to \infty} \mu_n(A_n) \geq \liminf_{l \to \infty} (1-2^{-l}) = 1.$$
	 Further, by definition of the sets $A_n$ and $A_{n,l}$, for a sequence $(\rho_n)$ with $\rho_n \in A_n$, we have
	\[
	\limsup_{n \to \infty} \big| \vartheta_n^{\rho_n}(p_i) - N^{\prime}(p_i) \big| < 1/l
	\]
	for all $l \in \N$ and $p_i \in S$ with $i \leq l$. This shows that $\lim_{n} \vartheta_n^{\rho_n}(p) = N^{\prime}(p)$ for all $p \in S$. Since $S$ is dense in $C(\Sigma)$, we obtain the desired assertion by a straightforward approximation argument.
\end{proof}

\section{A L\"uck type approximation theorem for local and empirical convergence} \label{sec:luck}

We now prove a L\"uck type approximation theorem asserting pointwise convergence of the spectral measures, and thus uniform convergence of the integrated density of states, see Theorem~\ref{thm:MAIN} and Corollary~\ref{cor:rationalvalues}.

\medskip

We say that a bounded operator $H: \ell^2(V) \to \ell^2(V)$ over a countable set $V$ {\em has integer coefficients} if $H(x,y) \in \mathbb{Z} + i \mathbb{Z}$ for all $x,y \in V$.

\medskip

The following lemma can be found in {\cite{abert-thom-virag}*{Lemma~9}}.  (In the latter paper, the authors deal with  slightly different operators, but the arguments of the proof carry over verbatim.) 
Similar estimates of the spectral measure of graph Laplacians and their local log-H\"older continuity property have been obtained also before, see e.g.\@ \cites{CS83,Luc94,Fab98,MY02,Ves05,thom}.


\begin{lem}\label{lemma-estimate}
Let $R \geq 1$. Then for every $\alpha \in [-R,R]$,  there exists a null sequence $(\varepsilon_k)$ of positive numbers such that for every finite set $V$ and for every self-adjoint linear operator $H:\ell^2(V) \to \ell^2(V)$  
that has integer
coefficients  and satisfies $\|H\| \leq R$, one has
\[
 \vartheta\big(]\alpha-\varepsilon_k,\alpha+\varepsilon_k[\setminus \{\alpha\}\big) \le \left(\frac{\log(2R)}{\log(\frac{1}{2\varepsilon_k})}+\frac{1}{k}\right)^{\frac{1}{2}},\]
where $\vartheta$ is defined according to~\eqref{eqn:tracemeasure}. 
\end{lem}

\begin{proof}
See the proof of Lemma~9 in \cite{abert-thom-virag}. 
\end{proof}

We are now in position to prove the main theorem of this section.

\begin{thm}[Uniform convergence of the IDS] \label{thm:MAIN}
    Let $G$ be a finitely generated group along with an ergodic process $(G,\mathcal{X},\mu)$, and let 
	 $(V_n,\sigma_n,\mu_n)$ be a sofic model. Let $\{H^{\omega}\}$ be a continuous operator    over $G$ that has integer coefficients. Then for every induced approximation $(\{H_n^{\rho}\})_n$, the following holds:   for all $\varepsilon > 0$ and each $\alpha \in \R$, 
	 \begin{align*}
	 	\lim_{n \to \infty} \mu_n \big( \big\{ \rho \in \mathcal{X}^{V_n}:\, |\vartheta_n^{\rho}(\{\alpha\}) - N^{\prime}(\{\alpha\})| < \varepsilon \big\} \big) = 1. 
	 \end{align*}
 In particular, there is a sequence $(A_n)$ with $A_n \subseteq \mathcal{X}^{V_n}$ and $\lim_n \mu_n(A_n) = 1$ such that for every sequence $(\rho_n)$ with $\rho_n \in A_n$, we obtain 
 \begin{align} \label{eqn:IDSconv}
	\lim_{n \to \infty} \sup_{\beta \in \R} \left| \frac{N_n^{\rho_n}(\beta)}{|V_n|} - N(\beta) \right| = 0,
\end{align}
where $N_n^{\rho}(\beta) = |V_n| \cdot \vartheta^{\rho}_n \big( ]-\infty,\beta] \big)$ is the number of eigenvalues of $H^{\rho}_n$ less or equal than $\beta$, counted with their multiplicities. 
\end{thm}

\begin{proof}
Suppose that $\{H^{\omega}\}$ is a continuous operator with coefficients in $\Z + i \Z$ and suppose that $(\{H_n^{\rho}\})_n$ is an induced approximation. 
	Let $\alpha \in \mathbb{R}$ and $\eta > 0$. 
	By Lemma \ref{lemma-estimate}, there exists a sequence $(\varepsilon_k)$ of positive numbers with  $0 < \varepsilon_k < \frac{1}{k}$ such that for all $n\in\N$ and $\rho\in \cX^{V_n}$ we have
	\[\vartheta^{\rho}_n({I}_k)\le \left(\frac{1}{k}+\frac{\log(2R)}{\log(\frac{1}{2\epsilon_k})}\right)^{\frac{1}{2}},\]
	where ${I}_k:=]\alpha-\varepsilon_k,\alpha+\varepsilon_k[\setminus\{\alpha\}$ and $R \geq 1$ is chosen (cf.\@ Remark~\eqref{rem:bounded}) such that
	\[
	\sup_{n \in \N} \sup_{\rho \in \mathcal{X}^{V_n}}\|H^\rho_n\| \leq R.
	\] 
	We now fix $k \in \mathbb{N}$ large enough such that $\vartheta_n^{\rho}({I}_k) \leq  \eta$ for all $n \in \mathbb{N}$ and $\rho \in \mathcal{X}^{V_n}$. Further, we choose a continuous function $h\in C_b(\R)$ of compact support with $1_{\{\alpha\}}\le h\le 1_{I_k \cup \{\alpha\}}$ and such that 
\[\int_{\mathcal{X}^G} \vartheta^{\omega}(h)d\mu(\omega) \le \int _{\mathcal{X}^G} \vartheta^{\omega}(\{\alpha\})d\mu(\omega) + \eta/2.\]
For $n \in \N$, let 
\[M_n:= \left\{\rho\in \mathcal{X}^{V_n}\;:\;\;\Big|\vartheta^{\rho}_n(h)-\int_{\mathcal{X}^G}\vartheta^{\omega}(h)d\mu(\omega) \Big|<\eta/2\right\}.
\] 
Note that since $\{H^{\omega}\}$ has integer coefficients, it is of bounded interaction, see the fourth bullet point of Remark~\ref{rem:boundedinteraction}. 
We deduce from Theorem ~\ref{lemma-weakconvergence} that $\lim_n \mu_n(M_n) = 1$.

By applying the condition of the set $M_n$ and the properties of $h$ we obtain for all $n\in \N$ and $\rho\in M_n$ that
\begin{align*}
\vartheta_n^\rho(\{\alpha\})&\le \vartheta_n^\rho(h)\le \int_{\mathcal{X}^G}\vartheta^{\omega}(h)d\mu(\omega) +\eta/2\le \int_{\mathcal{X}^G}\vartheta^{\omega}(\{\alpha\})d\mu(\omega) +\eta \\ &\le \int_{\mathcal{X}^G}\vartheta^{\omega}(h)d\mu(\omega) +\eta \le \vartheta_n^\rho(h) + \frac{3\eta}{2}\le \vartheta_n^\rho(I_{k} \cup \{\alpha\})+\frac{3\eta}{2}\\
&\le \vartheta^{\rho}_n(\{\alpha\})+\vartheta^{\rho}_n({I}_k)+\frac{3\eta}{2}\le \vartheta_n^{\rho}(\{\alpha\})+2\eta.
\end{align*}
Hence, we have shown that 
\[\Big\{ \rho \in \mathcal{X}^{V_n}:\, \Big|\vartheta_n^\rho(\{\alpha\})-\int_{\mathcal{X}^G}\vartheta^{\omega}(\{\alpha\})d\mu(\omega) \Big|\le \eta \Big\}\supset M_n\] and since $\lim_n \mu_n(M_n) = 1$, this concludes the proof.

\medskip

We now turn to the ``in particular''-part.  We follow the strategy of the proof of Corollary~\ref{lemma-weakconvergence-consequence}, building on the first part of the theorem above. To this end,  
	 fix a compact set $\Sigma \subseteq \R$ containing the supports of all measures $\vartheta_n^{\rho}$ and $N^{\prime}$, along with a countable dense set $S \subseteq C(\Sigma)$. 
	Note that the set of all  $\alpha \in \R$ with $N^{\prime}(\{\alpha\}) > 0$ is countable. Consequently, we can fix an arbitrary enumeration $i \mapsto p_i$, where $p_i$ is either in $S$ or the characteristic function $1_{\{\alpha\}}$ of some $\alpha$ with $N^{\prime}(\{\alpha\}) > 0$. For $l \in \N$ and $i \leq l$, we set  
		\[
	A_{n,l}^{(i)} := \big\{ \rho \in \mathcal{X}^{V_n}:\, |\vartheta_n^{\rho}(p_i) - N^{\prime}(p_i) | < 1/l \big\}.
	\]
	Now we can proceed exactly as in the proof of Corollary~\ref{lemma-weakconvergence-consequence}, to find $N_l$ such that for $n \geq N_l$ we have $\mu_n(A_{n,l}^{(i)}) > 1-2^{-l}/l$ (due to Theorem~\ref{lemma-weakconvergence} if $p_i \in S$ and due to the first part of the present theorem if $p_i = 1_{\{\alpha\}}$). As in the proof of Corollary~\ref{lemma-weakconvergence-consequence}, this gives a sequence $(A_n)$ of subsets in $\mathcal{X}^{V_n}$ with $\lim_{n \to \infty}\mu_n(A_n)=1$ and such that for every sequence $(\rho_n)$ with $\rho_n \in A_n$, the measures $(\vartheta_n^{\rho_n})$ converge weakly to $N^{\prime}$ and at the same time $\lim_{n \to \infty} \vartheta_n^{\rho_n}(\{\alpha\}) = N^{\prime}(\{\alpha\})$ for each $\alpha \in \R$ with $N^{\prime}(\{\alpha\}) > 0$. Moreover, for $\alpha \in \R$ with $N^{\prime}(\{\alpha\}) = 0$, weak convergence together with the Portmanteau lemma \cite{kechris2012classical}*{Chapter II, Theorem 17.20} for closed sets imply
	\[
	0 \leq \limsup_{n \to \infty} \vartheta_n^{\rho_n}(\{\alpha\}) \leq N^{\prime}(\{\alpha\}) = 0.  
	\]
	If follows from \cite{LV09}*{Lemma~6.3} that weak convergence, together with pointwise convergence of the measures in $(\vartheta^{\rho_n}_n)$ to $N^{\prime}$ imply  uniform convergence of their underlying distribution functions to the integrated density of states $N$. This finishes the proof. 
\end{proof}

In the following, We say that $\{H^{\omega}\}$ {\em takes finitely many values} if
\[
\{H^{\omega}(x,y): \omega \in \mathcal{X}^G, x,y \in G \}
\]
is a finite subset of $\C + i\C$. If all the values belong to $\Q + i \Q$, then we say that $\{H^{\omega}\}$ {\em takes finitely many rational values}.

\begin{cor} \label{cor:rationalvalues}
The statement of the previous theorem also holds if instead of integer coefficients, the family $\{H^{\omega}\}$ takes finitely many rational values. 
\end{cor}

\begin{proof}
By the assumption on the coefficients, there is an integer $q \in \mathbb{N}$ such that the family $\{ H^{\omega}_{*} \} = \{ qH^{\omega} \}$ has coefficients in $\Z + i \Z$ and satisfies all additional properties to apply Theorem~\ref{thm:MAIN}.
The proof is concluded after observing through the spectral calculus that the integrated density of states $N_{*}$ of $\{ H^{\omega}_{*} \} $ and the eigenvalue counting functions $N_{n,*}^{\rho}$  of its induced approximations have the scaling property
\[
N_{*}(q\beta) = N(\beta) \quad \quad \mbox{and} \quad\quad N_{n,*}^{\rho}(q\beta) = N_{n}^{\rho}(\beta), 
\]
for every $\beta \in \R$, and similary for the density of states measures.
\end{proof}

\begin{cor}[Uniform approximation of the IDS for PA groups] \label{thm:IDS_PA}
Suppose that $G$ is a finitely generated PA group and let $(G,\mathcal{A},\mu)$ be an ergodic $G$-process, where $\mathcal{A}$ is a finite set. Then there is a residually finite model $(V_n,\sigma_n,\mu_n)$ with the following property: for every continuous operator $\{H^{\omega}\}$ taking finitely many rational values, there exist an induced approximation $(\{H^{\rho}_n\})_n$ and a sequence $(A_n)$ with $A_n \subseteq \mathcal{A}^{V_n}$ with $\lim_{n \to \infty}\mu_n(A_n) = 1$ such that for $(\rho_n)$ with $\rho_n \in A_n$,
\begin{align*}
	\lim_{n \to \infty} \sup_{\beta \in \R} \left| \frac{N_n^{\rho_n}(\beta)}{|V_n|} - N(\beta) \right| = 0.
\end{align*}
\end{cor}

\begin{proof}
	We obtain a  residually finite approximation $(V_n,E_n,\sigma_n)$, along with $\mu_n \in \mathrm{Prob}(\mathcal{A}^{V_n}, V_n)$ that locally and empirically converge to $\mu$ from Corollary~\ref{cor:PAEPA}, in combination with Proposition~\ref{prop:leergodic}. An induced approximation $(\{H_n^{\rho}\})_n$ over the graphs $(V_n,E_n)$  exists due to Proposition~\ref{prop:goodapprox}. The statement now follows from Corollary~\ref{cor:rationalvalues}. 
	\end{proof}

\section{Rational approximation via operator monotone convergence} \label{sec:monotoneOP}

We are now going to transfer the concept of monotone convergence of operators to the realm of continuous operators. For random operators taking finitely many complex coefficients, we establish pointwise  approximation of the IDS via finite volume analogs with varying rational coefficients, see Corollary~\ref{cor:MAIN_approx}.

\medskip

Let $\{H^{\omega}\}$ be a continuous operator taking finitely many values and let $(\{H_m^{\omega}\})$ be a sequence of continuous operators. We say that $(\{H_m^{\omega}\})$ 
is {\em adapted to} $\{H^{\omega}\}$ if
\begin{itemize}
    \item there is a common finite hopping range parameter $M \in \N$ for $\{H^{\omega}\}$ and all $\{H^{\omega}_m\}$,
    \item all $\{H^{\omega}_m\}$ take finitely many values such that for $ x,y,w,z \in G, \, \omega,\rho \in \mathcal{X}^G$, one has 
    \[ 
     H^{\omega}(x,y) = H^{\rho}(w,z) \quad \Longrightarrow \quad H^{\omega}_m(x,y) = H^{\rho}_m(w,z).
     \]
     {for all $m\in \N$.}
     \item for $ x,y \in G, \, \omega \in \mathcal{X}^G$ and $m \in \N$, one has 
    \[
    H_m^{\omega}(x,y) \neq 0 \quad \iff \quad H^{\omega}(x,y) \neq 0.
    \]
\end{itemize}

\begin{defn}[Monotone operator convergence for continuous operators]
For a continuous operator $\{H^{\omega}\}$ taking finitely many values,
we say that a sequence $(\{H^{\omega}_m\})$ of continuous operators {\em operator monotonically converges} to  $\{H^{\omega}\}$ if 
\begin{itemize}
    \item $(\{H^{\omega}_m\})$ is adapted to $\{H^{\omega}\}$,
    \item $\lim_{m \to \infty} H^{\omega}_m(w,v) = H^{\omega}(w,v)$ for all $v,w \in G$ and every $\omega \in \mathcal{X}^G$,
    \item the operators and $H^{\omega}_{m+1} - H^{\omega}_m$ are positive semi-definite for all $m \in \N$ and every $\omega \in \mathcal{X}^G$.
\end{itemize}
If all of the latter operators  $H^{\omega}_{m+1} - H^{\omega}_m$ are even positive definite, we say that 
 $(\{H^{\omega}_m\})$ {\em strictly operator monotonically converges} to $\{H^{\omega}\}$ or just that the convergence of  $(\{H^{\omega}_m\})$ to $\{H^{\omega}\}$ is {\em strict}.
\end{defn}

\begin{rem}
A couple of remarks on the definition are in order. 
\begin{itemize}
    \item The second bullet point guarantees convergence of the operators $H^{\omega}_m$ to $H^{\omega}$ in the weak operator topology.
    \item Combined with the aforementioned convergence, the adaptedness condition yields that the convergence of coefficients can be described by finitely many convergent complex sequences, uniformly over all $\omega \in \mathcal{X}^G$. Via a straight forward computation one then even obtains convergence  in operator norm, as stated in Lemma~\ref{lem:montoneOPconvergence} below. 
    \item One might argue that the terminology should reflect the monotonically increasing order given in the definition. Since we will exclusively be interested in monotonically increasing convergence, we will just speak of monotone convergence.   
\end{itemize}
\end{rem}

 \begin{exmp} \label{exa:Schrödingerpert}
 A special case of interest are  operators being perturbed by potentials in a monotonically increasing way.
To this end, let $\mathcal{X}=\mathcal{A}$ be finite and let 
$F: \mathcal{A} \to \R$ be a map. For a continuous operator $\{H^{\omega}\}$ taking finitely many values, we consider its {\em perturbation by $F$} as 
\[
 H_F^{\omega}: \ell^2(G) \to \ell^2(G), \quad H_F^{\omega}u(x) = H^{\omega} u(x) + F(\omega(x)) u(x), \quad \omega \in \mathcal{A}^G.
\]
Clearly, $\{H_F^{\omega}\}$ gives rise to a continuous operator as well. It is straightforward to see that if $(F_m)$ is a sequence of maps $\mathcal{A} \to \R$ converging monotonically increasingly to $F:\mathcal{A} \to \R$, then $(\{H_{F_m}^{\omega}\})$  operator monotonically converges to $\{H^{\omega}_F\}$. 
The operator monotone convergence is strict if $(F_m)$ strictly monotonically converges to $F$. Note that the Schr\"odinger operators from Example~\ref{exa:Schrodinger} are perturbations of the discrete Laplacian by potentials $F:\mathcal{A} \to \R$.
\end{exmp}

Combining adaptedness and the convergence in weak operator topology yields convergence in operator norm, as demonstrated in the next lemma. 

\begin{lem} \label{lem:montoneOPconvergence}
    Suppose that $(\{H^{\omega}_m\})$  operator monotonically converges to a continuous $\{H^{\omega}\}$ taking finitely many values. Then the following holds. 
    \begin{itemize}
        \item[(i)] We have uniform convergence in operator norm, i.e.\@
         \begin{align*}
    \lim_{m \to \infty} \sup_{\omega \in \mathcal{A}^G} \big\| H^{\omega}_m - H^{\omega} \big\|_{\operatorname{op}} \, = \, 0.
    \end{align*}
    In particular, 
    \begin{align*}
    \sup_{m \in \N} \sup_{\omega \in \mathcal{A}^G} \big\| H_m^{\omega} \big\|_{\operatorname{op}} \, < \, \infty.
    \end{align*}
        \item[(ii)]   For each $\omega \in \mathcal{X}^G$, the spectral measures $\vartheta_m^{\omega}$ of $H_m^{\omega}$ converge weakly to the spectral measure $\vartheta^{\omega}$ of $H^{\omega}$, uniformly in $\omega \in \mathcal{X}^G$, i.e.\@ for every $f \in C_b(\R)$, 
        \begin{align*}
        \lim_{m \to \infty} \sup_{\omega \in \mathcal{X}^G} \big| \vartheta_m^{\omega}(f) - \vartheta^{\omega}(f) \big| = 0. 
        \end{align*}
    \end{itemize}

\end{lem}

\begin{proof}
    The uniform convergence over $\omega \in \mathcal{X}^G$ in operator norm is an immediate consequence from the observations made in the previous remark. The uniform boundedness of the operators follows from the uniform convergence in operator norm and the fact that by adaptedness, all $\{H_m^{\omega}\}$ take finitely many values. 

    As for the second assertion, note that for $p \in \R[x]$, we have convergence of $(p(H^{\omega}_m))_m$ to $p(H^{\omega})$ in operator norm, as follows from assertion~(i), together with the formula
\[
(H_{m}^{\omega})^s - (H^{\omega})^s = \sum_{i=0}^{{s-1}} (H_{m}^{\omega})^{s-1-i} \big(H_{m}^{\omega} -  H^{\omega}\big)(H^{\omega})^i
\]
for $s\in \N$. The weak convergence for the spectral measures now follows by a Stone-Weierstra{\ss} approximation argument. 
\end{proof}

We now state a continuity result of the IDS with respect to the coefficients of continuous operators. Note that this result does not hinge on soficity of the group. 

\begin{thm}[Convergence of the IDS for monotone sequences of continuous operators] \label{thm:OPMON}
Let $G$ be a finitely generated (not necessarily sofic) group, and let $(G,\mathcal{X},\mu)$ be an ergodic process.
Suppose that $\{H^{\omega}\}$ is a continuous operator over $G$ taking finitely many  values.  Further, assume that  $(\{H^{\omega}_m\})$ is a sequence of continuous operators over $G$ that 
 operator monotonically converges to  $\{H^{\omega}\}$. Then  
\begin{align} \label{eqn:IDScontinuous}
\lim_{m \to \infty} N_m(\beta) = \inf_{m \in \N} N_m(\beta) =N(\beta) \quad \mbox{ for all } \quad \beta \in \R,
\end{align}
 where $N_m$ and $N$ denote the integrated densities of states for $\{H_m^{\omega}\}$ and $\{H^{\omega}\}$, respectively. 
\end{thm}

 \begin{rem} \label{rem:notuniform}
 We make some remarks. 
 \begin{itemize}
     \item We point out that the monotonicity assumption on the convergence  is essential. The reason behind this is the fact that the integrated density of states is in general only continuous from the right.
     This becomes apparent already in very elementary settings of perturbations. For instance, assume that $\mathcal{X} =\cA = \{a\}$ consists of one element only (i.e.\@ $\mu$ is a point measure). Now for an arbitrary Cayley graph over some group $G$ and $m \in \N$, consider $H_{F_m}$ as the perturbation of the zero operator by the (constant) function $F_m$  determined by  $F_m(a)=1/m$. In simpler words, $H_{F_m}$ is just the multiplication operator by $1/m$. Clearly, $(F_m)$ converges to the zero function $F$ determined by $F(a) = 0$, but in a monotonically decreasing way. Since $H_F$ is the zero operator, we have $N_F(0)=1$. On the other hand, we have $N_{F_m}(0)=0$ for all $m \in \N$, so convergence fails in $\beta =0$. 
     \item In general, uniform convergence in $\beta$ cannot be expected. To see this, we modify the above example. We consider the functions $F_m$ determined by $F_m(a) = -1/m $, and $F$ is again the zero function. Now clearly, $(F_m)$ converges monotonically increasingly to the zero function $F$, and pointwise convergence of the IDS follows from the above theorem. However, for each $m \in \N$, we have 
     \[
    \sup_{\beta \in \R} \big| N_{F_m}(\beta) - N_{F}(\beta)\big| \geq \big| N_{F_m}(-1/m) - N_F(-1/m) \big| = 1 - 0 = 1,
     \]
     and uniform convergence fails to hold.    
     \item Continuity of the IDS with respect to varying Hamiltonians was also investigated in  \cite{LSV11}*{Section~5.5}. There, the authors prove  weak convergence of the density of states measure in the situation of an amenable Cayley graph whose vertices are decorated by a coloring having sufficient  ergodicity assumptions. Note that weak convergence can be obtained without monotonicity assumptions on the operator sequence.
 \end{itemize}
 \end{rem}

The proof of the theorem builds on a monotonicity criterion of finite traces in von Neumann algebras. We will use it as a black box. Recall that a {\em  von Neumann algebra} is a unital $*$-sub algebra of the space $\mathcal{B}(\mathcal{H})$ of bounded operators on a Hilbert space $\mathcal{H}$ that is closed in the weak operator topology. We will exclusively deal with the situation that $\mathcal{H}$ is separable since this is the generality we need and avoids some subtleties of the underlying theory. 
Denote by $\mathcal{M}_{+}$ the subset of positive semi-definite elements in a von Neumann algebra $\mathcal{M}$.  A {\em weight} $\tau$ on $\mathcal{M}$ is a  map $\tau:\mathcal{M}_{+} \to [0, \infty]$ that is linear on $\mathcal{M}_{+}$. A weight $\tau$ is called {\em finite} if  $\tau(\operatorname{Id}) < \infty$ for the identity operator $\operatorname{Id}$ on $\mathcal{H}$. A weight $\tau$ is called {\em faithful} if $\tau(a) = 0$ forces $a=0$ for $a \in \mathcal{M}_{+}$. Moreover, a weight $\tau$ is called {\em normal} if $\lim_{m \to \infty} \tau(A_m) = \tau(A)$ for every sequence $(A_m)$ in $\mathcal{M}_{+}$ that monotonically converges to $A$ in the strong operator topology, i.e.\@ $\lim_{m \to \infty} \|A_mx - Ax\|_{\mathcal{H}} = 0$ for all $x \in \mathcal{H}$ and $A_m - A_{\ell}$ is positive semi-definite whenever $m \geq \ell$. (Here, it is crucial to work with separable Hilbert spaces. In the non-separable case, one would have to work with nets.) A weight is called a {\em trace}  if it satisfies the {\em tracial property} $\tau(aa^{*}) = \tau(a^{*}a)$ for all $a \in \mathcal{M}$, where we denote by $a^{*}$ the adjoint of $a$.
For a deeper insight into operator algebras and their traces, we refer to the literature, e.g., \cites{Con79,kadisonringroseI, kadisonringroseII}.

\medskip
 For a process $(G,\mathcal{X},\mu)$ and continuous operators $A= \{A^{\omega}\}$ and $B= \{B^{\omega}\}$ over $G$ we say that $A$ and $B$ are {\em equivalent} if $A^{\omega} = B^{\omega}$ for $\mu$-almost every $\omega \in \mathcal{X}^G$. With a slight abuse of terminology and notation, we will not distinguish between representatives and classes. However we have to keep this distinction in mind since it becomes relevant in the following proposition.

 \medskip
The following proposition follows from (special cases of) results  from \cite{LPV07} and is tailored to what we need here. Note that for this we do not have to assume that $G$ is sofic.

\begin{prop} \label{prop:LPV}
Let $G$ be a finitely generated group (not necessarily sofic) and let $(G,\mathcal{Y},\mu)$ be an ergodic process. Then there is a von Neumann algebra $\mathcal{M}$ over a separable Hilbert space $\mathcal{H}$
    \begin{itemize}
        \item containing the equivalence classes of all continuous operators $\{H^{\omega}\}$ over $G$,
        \item admitting a faithful, finite and normal  trace $\tau:\mathcal{M}_{+} \to [0, \infty)$ such that for all $\beta \in \R$, 
        \begin{align*}
        \tau\big( 1_{]-\infty, \beta]}(\{H^{\omega}\}) \big) = N(\beta) := \int_{\mathcal{Y}^G} \langle \delta_e,\, 1_{]-\infty,\beta]}\big(H^{\omega}\big)\delta_e \rangle\, d\mu(\omega),
        \end{align*}
        where  $1_{]-\infty, \beta]}(\{H^{\omega}\}) \in \mathcal{M}_{+}$ denotes the corresponding spectral projection of $\{H^{\omega}\}$.
    \end{itemize}
\end{prop}

\begin{proof}
 We describe how our setting fits into the framework of \cite{LPV07}.
    Adopting the notation of the latter paper, we define $\mathcal{X} = \mathcal{G} = \Omega \times G$, where $\Omega = \mathcal{Y}^G$.
    Define the separable Hilbert space $\mathcal{H} = L^2(\mathcal{X}, \mu \times c)$, where $c$ denotes the counting measure on $G$. 
    Now every equivalence class of a continuous operator $A=\{H^{\omega}\}$ gives rise to a bounded random operator $A:L^2(\mathcal{X}, \mu \times c) \to L^2(\mathcal{X},\mu \times c)$ as defined in Section~3 of \cite{LPV07}, decomposing with direct integral theory measurably as $A^{\omega}:\ell^2(G) \to \ell^2(G)$, where we canonically identify $\ell^2(\mathcal{X}^{\omega}) \cong \ell^2(\{\omega\} \times G) \cong \ell^2(G) $ for $\omega \in \Omega$.
    (Note that in \cite{LPV07}, the action of $G$ on $\Omega$ is by left translations, while in this paper we consider right translations. Still, the results from \cite{LPV07} carry over by a simple reparametrization $\widetilde{\Omega}$ of configurations, given as $\widetilde{\omega}(h) = \omega(h^{-1})$, where $\omega \in \Omega$ and $h \in G$.) By \cite{Con79}*{Theorem~V.2}, the almost everywhere equivalence classes of bounded random operators define a von Neumann algebra $\mathcal{M} = \mathcal{N}(\mathcal{G},\mathcal{X})$ consisting of bounded, self-adjoint operators $A=\{A^{\omega}\}$ on $\mathcal{H}$,
    see also Theorem~3.1 in \cite{LPV07}. Sticking to the language of Section~3 in \cite{LPV07}, (classes of) continuous operators are bounded random operators which are affiliated to $\mathcal{M}$, meaning that all their spectral projections also belong to $\mathcal{M}$. By \cite{LPV07}*{Theorem~4.2~(a)}, applied to the map $u:\mathcal{X} \to [0,\infty), u(\omega,g) = \delta_{e}(g)$, we find that 
    \[
    \tau: \mathcal{M}_{+} \to [0, \infty), \quad \tau (A) = \int_{\Omega} \langle \delta_e,\, A^{\omega}\delta_e \rangle\, d\mu(\omega), \quad A = \{A^{\omega}\} \in \mathcal{M}_{+} 
    \]
    gives rise to a faithful, normal and finite weight (Note that the paper \cite{LPV07} shows semi-finiteness of $\tau$. Here, finiteness follows immediately from the fact that $\mu$ is a probability measure.) In order to justify that $\tau$ is even a trace, we use another criterion of the above mentioned paper: by Proposition~4.5 in \cite{LPV07}, the tracial property holds for all equivariant Carleman operators. Carleman operators are those that have a fibrewise representation by an $L^2$-kernel, see page~17 in \cite{LPV07} for a definition. However, in our situation, every bounded random operator $A = \{A^{\omega}\} \in \mathcal{M}$ is in fact an equivariant Carleman operator with kernel $k\big( (\omega,g), (\rho, h) \big) = \overline{((A^{\omega})^{*} \delta_{g})(h)}$ for $(\omega,g), (\rho, h) \in \mathcal{X}$. 
     The observation that for a continuous operator  $H= \{H^{\omega}\}$, we have 
    \[
\big(1_{]-\infty, \beta]}(H) \big)^{\omega} = 1_{]-\infty, \beta]}(H^{\omega})
    \]
    for $\mu$-almost every $\omega \in \Omega$ (as can be seen again via measurable decomposition), finishes the proof. 
\end{proof}

We can now give the proof of Theorem~\ref{thm:OPMON}.

\begin{proof}[Proof of Theorem~\ref{thm:OPMON}]
    By assumption on operator monotone convergence, we have that $H_{m+1}^{\omega} - H_m^{\omega}$ and $H^{\omega}- H_m^{\omega}$ are positive semi-definite for $m \in \N$ and $\omega \in \mathcal{X}^G$. Moreover, the (classes of) continuous operators $H_m = \{H_m^{\omega}\}$ and $H= \{H^{\omega}\}$ 
    belong to the von Neumann algebra $\mathcal{M}$ given by Proposition~\ref{prop:LPV}. We thus conclude that $H_{m+1} - H_m \in \mathcal{M}_{+}$ and $H - H_m \in \mathcal{M}_{+}$ for every $m \in \N$. Also, by Proposition~\ref{prop:LPV}, $\mathcal{M}_{+}$ admits a faithful, normal and finite trace $\tau$ such that for all $\beta \in \R$,  
    $\tau\big( 1_{]-\infty,\beta]}(H_m) \big) = N_m(\beta)$ and $\tau\big( 1_{]-\infty,\beta]}(H) \big) = N(\beta)$, with $N$ and $N_m$ denoting the integrated densities of states of $H$ and $H_m$, respectively. Since perturbations by constants just results in translations of the density of states measure, we will assume with no loss of generality that $H_m, H \in \mathcal{M}_{+}$. This will put us in the position to apply a result from the literature. So let $\beta \geq  0$. Setting $g = 1_{]\beta, \infty[} = 1- 1_{]-\infty, \beta]}$, we find by
     by Lemma~3~(i) and inequality~(6) in \cite{BK90}
    \begin{align*}
    N_{m+1}(\beta) &= \tau\big( (1-g)(H^{m+1}) \big) \leq  \tau\big( (1-g)(H^{m}) \big) = N_m(\beta), \\
    N(\beta) &= \tau\big( (1-g)(H) \big) \leq \tau\big( (1-g)(H_m) \big) = N_m(\beta), 
    \end{align*}
    for all $m \in \N$. Consequently, the limit $N^{*}(\beta):= \lim_{m \to \infty} N_m(\beta) = \inf_{m \in \N} N_m(\beta)$
    exists for all $\beta$.  It remains to identify the limit as $N = N^{*}$.  We already know that $N(\beta) \leq \inf_m N_m(\beta) = N^{*}(\beta)$. Further, 
    we deduce from Lemma~\ref{lem:montoneOPconvergence} that the spectral measures $\vartheta^{\omega}_m$ associated with $H_m^{\omega}$ converge weakly to the spectral measure $\vartheta^{\omega}$ associated with $H^{\omega}$, uniformly in $\omega$. 
By Portmanteau's lemma, cf.\@ \cite{kechris2012classical}*{Chapter II, Theorem 17.20}, weak convergence yields 
\[
N^{*}(\beta) = \limsup_{m \to \infty} N_{m}(\beta) \leq N(\beta).
\]
This finishes the proof.
\end{proof}

\begin{lem} \label{lem:monotoneOPdensity}
Let $\{H^{\omega}\}$ be a continuous operator  taking finitely many values. Then there is a sequence $(\{H_m^{\omega}\})_m$ of continuous operators  that   strictly operator monotonically converges to $\{H^{\omega}\}$, with each $\{H^{\omega}_m\}$ taking finitely many rational values. \\
Moreover, we can make sure that for each $m \in \N$, every $\omega \in \mathcal{X}^G$ and each $v \in G$, we have
\begin{align} \label{eqn:quantGrishgorin}
H_{m+1}^{\omega}(v,v) - H_{m}^{\omega}(v,v) \, > \, \sum_{w \neq v} \big| H_{m+1}^{\omega}(v,w) - H_{m}^{\omega}(v,w) \big|.
\end{align}
\end{lem}

\begin{rem}
Note that by Proposition~\ref{prop:Gershgorin}, the inequalities given in~\eqref{eqn:quantGrishgorin} are a stronger condition that the assumption that the operators $H_{m+1} - H_m$ are positive definite (which is a part of the definition of strict operator monotone convergence). 
\end{rem}

\begin{proof}[Proof of Lemma~\ref{lem:monotoneOPdensity}]
 By assumption, $\{H^{\omega}\}$ only takes finitely many values from $\C$.
    Consequently, 
    \begin{align*}
    F_1 &= \{ H^{\omega}(e,e):\, \omega \in \mathcal{X}^G \} \setminus \{0\} \\
    F_2 &=  \{ H^{\omega}(e,w):\, w \neq e,\, \omega \in \mathcal{X}^G \} \setminus \{0\}
    \end{align*}
    are finite subsets of $\C$, and by self-adjointness, $F_1 \subseteq \R$. For each $f \in F_1$, find a sequence $(a(f)_m)$ of rational, non-zero numbers converging strictly monotonically increasingly to $f$.
    Let now $h \in F_2$. We assume first that $\mathrm{Im}(h) \geq 0$. We find a sequence $(b(h)_m)$ of rational non-zero numbers such that $\mathrm{Re}(b(h)_{m+1}) \leq \mathrm{Re}(b(h)_m)$, $\mathrm{Im}(b(h)_{m+1}) \leq \mathrm{Im}(b(h)_m)$  and such that $\lim_{m \to \infty} b(h)_m = h$ in such a way that for each $m \in \N$, we have
    \begin{align} \label{eqn:ldfun}
    a(f)_{m+1} - a(f)_m &> 2\sum_{h \in F_2, \mathrm{Im}(h) \geq 0} \big|  b(h)_{m} - h \big| \nonumber \\ 
    &\geq  2 \sum_{h \in F_2, \mathrm{Im}(h) \geq 0} \big|  b(h)_{m} - b(h)_{m+1} \big|.
    \end{align}
    Note that  sequences $(b(h)_m)_m$ satisfying the first set of inequalities can be constructed recursively, since $a(f)_{m+1} > a(f)_m$ for all $m \in \N$ and all $f \in F_1$. The second inequality follows from the monotonicity assumptions on $(b(h)_m)$. If $h \in F_2$ with $\mathrm{Im}(h) < 0$, we simply set $b_m(h) = \overline{b_m(\overline{h})}$. 
    We now set for $x,y \in G$ with $yx^{-1} \in B_M$
    \begin{align} \label{eqn:construction}
     H_m^{\omega}(x,y) = 
    \begin{cases}
    a_m(f), & x=y,\, H^{x\omega}(e,e) = f, \\
    b_m(h), & x \neq y,\, H^{x\omega}(e, yx^{-1}) = h,
    \end{cases}
    \end{align}
    and $H^{\omega}_m(x,y) = 0$ if $yx^{-1} \notin B_M$. It is clear that this gives rise to a continuous operator. We just check equivariance, the other properties can be verified in a similar way. To this end, fix $x,y,g \in G$. We assume that $yx^{-1} \in B_M$ and $x \neq y$, the other cases can be treated similarly. Indeed, we have $H_m^{g\omega}(x,y) = b_m(h) $, with $H^{xg\omega}(e,yx^{-1}) = h$. But we also have $H_m^{\omega}(xg,yg) = b_m(h)$, since $H^{xg\omega}(e, yg(xg)^{-1}) = H^{xg\omega}(e,yx^{-1}) = h$. The self-adjointness follows from the construction of the sequence $(b_m(h))$ with $b_m(\overline{h}) = \overline{b_m(h)}$ for $f \in F_2$. Also, it follows from~\eqref{eqn:ldfun} that for all $m \in \N$, 
    \begin{align*}
      a(f)_{m+1} - a(f)_m  
    \,>\,  \sum_{h \in F_2} \big|  b(h)_{m} - b(h)_{m+1} \big|.
    \end{align*}
    By construction, the sequence $(\{H^{\omega}_m\})$ is adapted to $\{H^{\omega}\}$. 
    Together with Proposition~\ref{prop:Gershgorin}, 
    this finishes the proof.
\end{proof}
\medskip

\begin{cor} \label{cor:MAIN_approx}
Let $(G,\mathcal{X},\mu)$ be an ergodic process, and let $(V_l, \sigma_l, \mu_l)$ be a sofic model. 
Assume that $\{H^{\omega}\}$ is a continuous operator over $G$ taking finitely many values. Then there are a sequence $(\{H^{\omega}_m\})$  adapted to $\{H^{\omega}\}$ of continuous operators taking finitely many rational values  and a sequence $(A_l)$ with $A_l \subseteq \mathcal{X}^{V_l}$ and $$\lim_{l \to \infty} \mu_l(A_l) = 1$$ such that for each sequence $(\rho_l)$ with $\rho_l \in A_l$, and every $\beta \in \R$,
\begin{align*}
N(\beta) &= \lim_{m \to \infty} \lim_{l \to \infty} \frac{N_m^{\rho_l}(\beta)}{|V_l|} = \inf_{m \in \N} \lim_{l \to \infty} \frac{N_m^{\rho_l}(\beta)}{|V_l|},
\end{align*}
where $N$ is the IDS of $\{H^{\omega}\}$ and $N_m^{\rho}$ denotes the eigenvalue counting function corresponding to the induced approximation of $\{H^{\omega}_m\}$ given by Proposition~\ref{prop:goodapprox}. \\
In particular, there is a subsequence $(V_{l_m}, \sigma_{l_m},\mu_{l_m})$ of the sofic model such that for every sequence $(\rho_{m})$ with $\rho_{m} \in A_{l_m}$,  
\[
N(\beta) = \lim_{m \to \infty} \frac{N_m^{\rho_{m}}(\beta)}{|V_{l_m}|} \quad \mbox{ for each } \quad \beta \in \R. 
\]
\end{cor}

\begin{proof}
By Lemma~\ref{lem:monotoneOPdensity}, there is a sequence $(\{H^{\omega}_m\})$ of continuous operators taking finitely many rational values and operator monotonically converging to $\{H^{\omega}\}$. By Theorem~\ref{thm:OPMON},  $\lim_{m \to \infty} N_m(\beta) \\ = N(\beta)$ for all $\beta \in \R$. As before, we apply Theorem~\ref{thm:MAIN} together with a  diagonal sequence argument similar to the ones used in the proofs of Corollary~\ref{lemma-weakconvergence-consequence} and Theorem~\ref{thm:MAIN}: given $l,m \in \N$, define the sets 
\begin{align*}
A_{l,m}: &= \big\{ \rho \in \mathcal{X}^{V_l}:\, | \vartheta^{\rho}_{k,l}(f^k_j) - N^{\prime}_k(f^k_j) | < \frac{1}{m}, \quad | \vartheta^{\rho}_{k,l}(h_s) - N^{\prime}_k(h_s) | < \frac{1}{m},  \\
& \quad\quad\quad | \vartheta^{\rho}_{k,l}(p_i) - N^{\prime}_k(p_i) | < \frac{1}{m} \quad \mbox{ for all } \quad 1 \leq i,j,k, s\leq m  \big\},
\end{align*}
where $i \mapsto p_i$ is an enumeration of a countable dense set $S \subseteq C(\Sigma)$, and $\Sigma \subseteq \R$ is compact supporting the spectra of all involved operators, and for each $k \in \N$, $j \mapsto f^k_j$ is an enueration of all $1_{\{\alpha\}}$ with $N_k^{\prime}(\{\alpha\}) > 0$,
and $s \mapsto h_s$ is an enueration of all $1_{\{\beta\}}$ with $N^{\prime}(\{\beta\}) > 0$.
These sets are finite intersection of sets $B_l$ with $\lim_{l \to \infty}\mu_l(B_l)$ = 1 by
Theorem~\ref{thm:MAIN}, and so as in the proof of Theorem~\ref{thm:MAIN}, we find a sequence $(A_l)$ with $\lim_l\mu_l(A_l) = 1$ such that for each $(\rho_l)$ with $\rho_l \in A_l$, every $m \in \N$ and every $\beta \in \R$, we obtain 
\[
\lim_{l \to \infty} \frac{N_m^{\rho_{l}}(\beta)}{|V_l|} = N_m(\beta). 
\]
We conclude the argument for the first statement with the convergence statement mentioned at the beginning of the proof. As for the ``in particular''-part, for each  $K \in \N$, pick 
$m = m(K) \geq K$ and 
$l_m \in \N$ large enough such that for $l \geq l_m$, we have $\mu(A_{l,m}) > 1 - 1/K$ and $|N^{\prime}_m(g_s) - N^{\prime}(g_s)| < 1/K$, as well as $|N^{\prime}_m(p_i) - N^{\prime}(p_i)| < 1/K$ for all $1 \leq i,s \leq K$, where $s \mapsto g_s$ is an enumeration of all $1_{]-\infty,\beta]}$ with $N^{\prime}(\{\beta\}) > 0$. This way, we obtain a subsequence $(l_m) = (l_{m(K)})_{K}$ and with the definition of the sets $A_{l,m}$, 
the claim now follows for the subsequence $(V_{l_m}, \sigma_{m_l}, \mu_{m_l})$.  
\end{proof}

We return to the random Schr\"odinger operator $\{H_F^{\omega}\}$ from Example~\ref{exa:Schrodinger}, defined via perturbations as of Example~\ref{exa:Schrödingerpert} through potentials $F:\mathcal{A} \to \R$, with $\mathcal{A}$ being a finite set.

\begin{cor} \label{cor:Schrodinger}
Let $G$ be a finitely generated PA-group and let $(G,\mathcal{A},\mu)$ be an ergodic $G$-process, where $\mathcal{A}$ is a finite set. Consider the random Schr\"odinger operator with potential $F:\mathcal{A} \to \R$. Then for every monotonically increasing sequence $(F_n)$ of maps $F_n:\mathcal{A} \to \Q$ converging pointwise to $F$, there is a sofic model $(V_n,\sigma_n,\mu_n)$ and a sequence $(A_n)$ with $A_n \subseteq \mathcal{A}^{V_n}$, $\lim_{n} \mu_n(A_n) =1$ such that for every sequence $(\rho_n)$ with $\rho_n \in A_n$, we have
\begin{align*}
\lim_{n \to \infty} \frac{N^{\rho_n}_{F_n,n}(\beta)}{|V_n|} = N_{F}(\beta) \quad \mbox{ for all } \quad \beta \in \R,
\end{align*}
where $N_{F_n, k}^{\rho}$ denotes the eigenvalue couting function for the operator $\Delta + F_n \circ \rho$ on $\ell^2(V_k)$.
\end{cor}

\begin{proof}
   Since the maps $F_n$ converge pointwise and monotonically increasingly to $F$, the associated continuous operators $\{H_{F_n}^{\omega}\}$ operator monotonically converge to $\{H_{F}^{\omega}\}$. Further, since $G$ is PA, there is a sofic model $(V_l,\sigma_l,\mu_l)$. 
   We can now follow  the same line of argumentation as in the proof of the previous corollary to  find  subsequences of the sofic model and of $(A_l)$, for simplicity denoted by $(V_{n}, \sigma_n, \mu_n)$ and $(A_n)$, satisfying what was claimed. 
\end{proof}

\section{Supplement: A Gershgorin type criterion for self-adjoint operators} \label{sec:supplementary}

In this supplementary section, we present a sufficient criterion for certain bounded, self-adjoint operators to be positive (semi-)definite. Assuming uniform absolute summability over the ``rows'', one can use a well-known method from the finite-dimensional case based on Gershgorin circles.

\medskip

Let $I$ be a countable set and denote by $A:\ell^2(I) \to \ell^2(I)$ a bounded operator. Recall that $A$ is called {\em positive semi-definite} if $\langle x,\,Ax \rangle \geq 0$ for all $x \in \ell^2(I)$. If the latter inequality is strict for all $x \neq 0$, then $A$ is called {\em positive definite}. 

\medskip

For $i,j \in I$ we define 
$a_{ij}:= \langle \delta_i,\, A\delta_j \rangle$.

 \begin{prop} \label{prop:Gershgorin}
Suppose further that the following conditions  hold on $A$. 
 \begin{enumerate}
    \item $A$ is self-adjoint, i.e\@ $a_{ji} = \overline{a_{ij}}$ for all $i,j \in I$;
     \item we have 
     \[
     C:=\sup_{i \in I} \sum_{j \in I} |a_{ij}| < \infty;
     \]
     \item for all $i \in I$, we have 
     \[
     a_{ii} \geq \sum_{j \neq i} |a_{ij}|.
     \]
 \end{enumerate}
 Then $A$ is positive semi-definite. If all inequalities in~(3) are strict, then $A$ is even positive definite.
 \end{prop}

 \begin{proof}
 Let $x \in \ell^2(I)$ be arbitrary. We observe first that 
 \begin{align*}
\langle x,Ax\rangle=\sum_{i\in I}\sum_{j\in I}\overline{x_i}\,a_{ij}\,x_j,
 \end{align*}
where the double sum converges absolutely, as can be seen from the computation
\begin{align*}
\sum_{i \in I} \sum_{j \in I} |\overline{x}_i a_{ij} x_j| &\leq  \sum_{i \in I} \sum_{j \in I} |a_{ij}| \frac{1}{2} \Big( |x_i|^2 + |x_j|^2 \Big) \\
&= \sum_{i \in I} \sum_{j \in I} |a_{ij}| |x_i|^2  \leq C\, \|x\|^2.
\end{align*}
Note that we used $2|x_i||x_j| \leq |x_i|^2 + |x_j|^2 $ in the first step, self-adjointness~(1) in the second  step and the condition~(2) in the last step. The real part can be decomposed into
\begin{align*}
\mathrm{Re} \, \langle x,\, Ax\rangle = \sum_{i \in I} a_{ii}|x_i|^2 \, + \, \sum_{i,j \in I, i \neq  j} \mathrm{Re}(\overline{x_i}a_{ij}x_j).
\end{align*}
With $\mathrm{Re}(\overline{x_i}a_{ij}x_j) \geq - |a_{ij}||x_i||x_j|$, we find
\begin{align*}
\mathrm{Re} \, \langle x,\, Ax\rangle  &\geq \sum_{i \in I} a_{ii}|x_i|^2  -  \sum_{i,j \in I, i \neq  j} |a_{ij}||x_i||x_j| \\
&= \sum_{i \in I} \left( a_{ii}|x_i|^2 - \sum_{j \in I, j \neq i} |a_{ij}||x_i||x_j| \right).
\end{align*}
A similar argument as the one used above for absolute convergence, based on self-adjointness and symmetrization with $2|x_i||x_j| \leq |x_i|^2 + |x_j|^2 $, shows that 
\[
\sum_{i \in I} \sum_{j \in I, j \neq i} |a_{ij}| |x_i| |x_j| \leq \sum_{i \in I} \Big( \sum_{j \in I, j \neq i} |a_{ij}| \Big) |x_i|^2. 
\]
We obtain
\[
\mathrm{Re}\, \langle x,\,Ax \rangle \, \geq \, \sum_{i \in I} \Big( a_{ii} - \sum_{j \in I, j \neq i} |a_{ij}| \Big) |x_i|^2.
\]
Now the assumption~(3) gives $\mathrm{Re}\, \langle x,\,Ax \rangle \geq 0$ and since $A$ is self-adjoint, we obtain $\langle x,\,Ax \rangle \geq 0$. In case that all inequalities in~(3) are strict, we obtain that $\langle x,\,Ax \rangle > 0$ whenever $x \neq 0$. This finishes the proof.
\end{proof}

\bibliography{refs}
\bibliographystyle{plain}
\end{document}